\documentclass[a4paper,11pt,reqno]{amsart}

\usepackage[T1]{fontenc}
\usepackage{lmodern}
\usepackage{amsmath,amssymb,amsthm,mathtools,mathrsfs}
\usepackage{geometry,hyperref,xcolor,microtype}

\allowdisplaybreaks

\definecolor{DarkBlue}{rgb}{0.1,0.1,0.55}
\definecolor{DarkRed}{rgb}{0.55,0.1,0.1}

\hypersetup{colorlinks=true,linkcolor=DarkBlue,citecolor=DarkRed,urlcolor=DarkBlue}

\numberwithin{equation}{section}
\theoremstyle{plain}
\newtheorem{theorem}{Theorem}[section]
\newtheorem{conjecture}[theorem]{Conjecture}
\newtheorem{proposition}[theorem]{Proposition}
\newtheorem{lemma}[theorem]{Lemma}

\newtheorem{remark}[theorem]{Remark}
\theoremstyle{definition}

\newcommand{\Sph}{\mathbb{S}}
\newcommand{\CP}{\mathbb{C}P}
\newcommand{\C}{\mathbb{C}}
\newcommand{\R}{\mathbb{R}}
\newcommand{\Nzero}{\mathbb{N}_0}
\newcommand{\dd}{\,d}
\newcommand{\rd}{\mathrm{rd}}
\newcommand{\Dom}{\operatorname{Dom}}
\newcommand{\Spec}{\operatorname{Spec}}
\newcommand{\Ran}{\operatorname{Range}}
\newcommand{\Id}{\operatorname{Id}}
\newcommand{\sourcecite}[2]{\cite[#2]{#1}}

\title[The Simon conjecture and curvature rigidities]{A proof of the Simon conjecture and curvature rigidities}

\author[Weiran Ding]{Weiran Ding$^{1}$}
\address{$^{1}$School of Mathematical Sciences, South China Normal University, Guangzhou 510000, P. R. China.}
\email{dingwr0806@m.scnu.edu.cn}

\author[Jianquan Ge]{Jianquan Ge$^{2}$}
\address{$^{2}$School of Mathematical Sciences, Laboratory of Mathematics and Complex Systems, Beijing Normal University, Beijing 100875, P. R. China.}
\email{jqge@bnu.edu.cn}

\author[Fagui Li]{Fagui Li$^{3,*}$}
\address{$^{3,*}$Frontier Interdisciplinary Domain, Beijing Institute of Technology, Zhuhai, Guangdong 519088, P. R. China.}
\email{lifagui@bitzh.edu.cn}

\date{}
\thanks{* Corresponding author.}

\begin{document}

\begin{abstract}
In this paper, we prove all gaps in the Simon conjecture for connected closed minimal surfaces immersed in unit spheres. Our proof uses the line-bundle ladder of Lin, Wang, and Xu and a Dolbeault computation to obtain two-sided ordered spectral comparisons on the $2$-sphere under Gaussian curvature $K_g\geq1$ and $K_g\leq1$, with rigidity in the equality case. Applying these comparisons at consecutive Calabi scales forces the Gaussian curvature to be one of the two endpoint values. We also establish the corresponding rigidity theorem for closed surfaces with parallel mean curvature vector. Finally, we prove a rotationally symmetric higher-dimensional analogue under a Ricci lower bound and a sectional-curvature upper bound.
\end{abstract}

\maketitle

\section{Introduction}\label{sec:introduction}

The study of rigidity and gap phenomena for minimal surfaces in spheres has a long history. In 1967, E. Calabi \cite{Calabi} classified linearly full minimal $2$-spheres of positive constant Gaussian curvature in a round sphere; do Carmo and Wallach~\cite{doCarmoWallach} later obtained the corresponding classification from a different viewpoint. For the unit ambient sphere, the possible positive constant curvatures are $2/[s(s+1)]$, where $s$ is a positive integer. Motivated by this classification, U. Simon \cite{Simon,Simon1} proposed a quantization problem for closed minimal surfaces in the unit sphere at the 1980 New Orleans Conference on Harmonic Maps. For $s=1,2,3,\cdots$, set
\begin{equation*}\label{eq:levels}
	\begin{aligned}
		K(s)&=\frac{2}{s(s+1)}=\frac{1}{T(s)},\\
		|A|^2(s)&=2-2K(s)=\frac{2(s-1)(s+2)}{s(s+1)}.\\
	\end{aligned}
\end{equation*}
Here $K(s)$ is the $s$-th Calabi curvature, $|A|^2(s)$ is the corresponding value of the squared norm $|A|^2$ of the second fundamental form, and $T(s)=K(s)^{-1}$. Throughout this paper, all immersions are assumed to be smooth and to have connected domains, $S^n$ denotes the smooth $n$-sphere, possibly equipped with a Riemannian metric $g$, whereas $\mathbb S^n(\kappa)$ denotes the round $n$-sphere of constant sectional curvature $\kappa>0$. We write $\mathbb S^n=\mathbb S^n(1)$. The Simon conjecture asks whether a closed linearly full minimal surface satisfying
\begin{equation*}
	K(s+1)\leq K\leq K(s)
\end{equation*}
must have $K\equiv K(s)$ or $K\equiv K(s+1)$. By the Gauss equation, this is equivalent to
\begin{equation*}
	|A|^2(s)\leq |A|^2\leq |A|^2(s+1).
\end{equation*}

\begin{conjecture}[Simon's conjecture]\label{simonconj}
	Let $M$ be a closed surface minimally immersed in  $\mathbb{S}^N$ such that the image is not contained in any hyperplane of $\mathbb{R}^{N+1}$. If $K(s+1)\leq K\leq K(s)$ for an $s\in \mathbb{N}$, then either $K\equiv K(s+1)$ or $K\equiv K(s)$, and thus the lifted immersion on the oriented double cover is one of Calabi's $2$-spheres with the dimension of the ambient space $N=2s+2$ or $N=2s$, respectively.
\end{conjecture}
The Simon conjecture belongs to the broader family of curvature-gap problems for minimal submanifolds in spheres, including the Chern and Lu conjectures; see the surveys \cite{GuXuXu,Yau}. Prior to the present work, the Simon conjecture had been proved only in the cases $s=1$ and $s=2$ \cite{BKSS,Simon}. These results show that
\begin{enumerate}
	\item if $\frac{1}{3}\leq K\leq 1$ then $K\equiv\frac{1}{3}$ or $K\equiv 1$;
	\item if $\frac{1}{6}\leq K\leq\frac{1}{3}$ then $K\equiv\frac{1}{6}$ or $K\equiv\frac{1}{3}$.
\end{enumerate}
Before the present work, the third gap problem in the Simon conjecture remained open, although many partial results were known (cf. \cite{Bolt88a,Bolt88b,It88,LiSimon}) for the case $s\ge3$. In our previous work \cite{DGL}, we gave a unified proof of the first two gaps $(s=1,2)$ by Simons-type integral identities and obtained a quantitative lower bound for the oscillation $\sup K-\inf K$ in the third interval $\frac{1}{10}\leq K\leq\frac{1}{6}$. That estimate gives a genuine separation phenomenon in the interior of the interval without imposing any additional assumptions on the normal bundle, but degenerates at its two endpoints. A subsequent refinement \cite{DingGeLiThird} established rigidity in one-sided neighborhoods of both endpoints, forcing $K\equiv\frac{1}{10}$ near the lower endpoint and $K\equiv\frac{1}{6}$ near the upper endpoint, and also improved the interior oscillation estimate. These results provided partial solutions of the third gap, but did not rule out all nonconstant configurations in the full interval $\frac{1}{10}\leq K\leq\frac{1}{6}$.\par
The argument developed here is different and applies to every positive integer $s$. We use the representation and min-max theorems for closed semibounded forms \cite{Kato}, elliptic regularity and Rellich compactness \cite{TaylorPDE}, unique continuation~\cite{Aronszajn}, uniformization \cite{JostRiemann}, and standard results from Dolbeault theory \cite{Forster,GriffithsHarris,Wells}. Our first result proves Conjecture \ref{simonconj}.
\begin{theorem}\label{thm:main}
	Let $x\colon M^2\to\Sph^N$ be a closed minimal immersion, and let $s$ be a positive integer. Suppose that the Gaussian curvature $K$ of $M$ satisfies $$\frac{2}{(s+1)(s+2)}\leq K\leq\frac{2}{s(s+1)}.$$ Then $$K\equiv\frac{2}{(s+1)(s+2)}\quad\text{or}\quad K\equiv\frac{2}{s(s+1)}.$$ Moreover, the lift of $x$ to the orientation double cover of $M$ is one of Calabi's $2$-spheres.
\end{theorem}

The main analytic ingredient is the following two-sided comparison for the ordered spectrum of a smooth Riemannian $2$-sphere. Combining Theorem \ref{thm:two-sided} and Takahashi's theorem \cite[Theorem 3]{Takahashi} we can give the proof of Theorem \ref{thm:main}.

\begin{theorem}\label{thm:two-sided}
	Let $g$ be a smooth metric on $S^2$, let $g_{\rd}$ be the unit round metric, and index eigenvalues from zero with multiplicity: $$0=\lambda_0(g)<\lambda_1(g)\leq\lambda_2(g)\leq\cdots.$$ Then, for every $i\geq1$,
	\begin{enumerate}
		\item\label{geq} if $K_g\geq1$, then $\lambda_i(g)\geq\lambda_i(g_{\rd})$;
		\item\label{leq} if $K_g\leq1$, then $\lambda_i(g)\leq\lambda_i(g_{\rd})$.
	\end{enumerate}
	Moreover, if equality holds in either situation for some $i_0\ge1$, then $(S^2,g)$ is isometric to $\Sph^2$.
\end{theorem}

Theorem \ref{thm:two-sided} \eqref{geq}, including the rigidity statement, was proved by Lin, Wang, and Xu \cite[Theorem 1.8]{LinWangXu}. For completeness, we give a self-contained presentation of their line-bundle argument. The same exact quadratic-form identity can be used in the opposite direction when $K_g\leq1$, yielding the second implication.  This reverse comparison is the additional ingredient needed to close the Simon gaps.\par
We next turn to surfaces with parallel mean curvature vector. For brevity, surfaces with parallel mean curvature vector will occasionally be called PMC surfaces. For a general immersion, write
\begin{equation*}
	\mathbf H=\frac12\operatorname{tr}_gA,\quad H=|\mathbf H|,\quad\kappa=1+H^2,
\end{equation*}
and let $A^\circ=A-g\otimes\mathbf H$ be the traceless second fundamental form with squared norm denoted by $|A^\circ|^2$. If $\mathbf H$ is parallel, then $H$ and $\kappa$ are constant. In our previous work \cite{DingGeLiPMC}, using Simons-type integral identities for $A^\circ$, we proved the first two PMC gaps and rigidity in one-sided neighborhoods of both endpoints of the third gap. The theorem below settles all gaps for PMC surfaces.

\begin{theorem}\label{thm:pmc-gaps}
	Let $x\colon M^2\to\Sph^N$ be a closed immersion with parallel mean curvature vector, and let $s$ be a positive integer. Suppose that the Gaussian curvature of $M$ satisfies
	\begin{equation}\label{eq:pmc-K-pinching}
		\frac{2\kappa}{(s+1)(s+2)}\leq K\leq\frac{2\kappa}{s(s+1)}.
	\end{equation}
	Then
	\begin{equation*}\label{eq:pmc-K-endpoints}
		K\equiv\frac{2\kappa}{(s+1)(s+2)}\quad\text{or}\quad K\equiv\frac{2\kappa}{s(s+1)}.
	\end{equation*}
	Moreover, the lift of $x$ to the orientation double cover of $M$ is one of Calabi's $2$-spheres in a totally umbilic sphere of $\Sph^N$.
\end{theorem}

The proof of Theorem \ref{thm:pmc-gaps} combines Theorem \ref{thm:main} with Yau's classification theorem; the three-dimensional branch is handled by the Hopf quadratic differential.\par
The radial version of the same mechanism gives a higher-dimensional result under rotational symmetry. On the lower side we use precisely the Ricci-curvature hypothesis of Lin, Wang, and Xu \cite[Theorem 6.5]{LinWangXu}, rather than the stronger sectional-curvature lower bound. For $n\geq3$ and $r\geq1$, put
\begin{equation*}
	\mu_{n,r}=r(r+n-1),\quad\kappa_{n,r}=\frac{n}{\mu_{n,r}}=\frac{n}{r(r+n-1)}.
\end{equation*}
Here and below, inequalities involving $\sec_g$ are understood pointwise on all tangent two-planes, while inequalities between Ricci tensors are understood pointwise as inequalities of symmetric $(0,2)$-tensors.

\begin{theorem}\label{thm:higher-rotational-gap}
	Let $n\geq 3$ and $r\geq 1$ be integers, and let $x\colon (S^n,g)\to\Sph^N$ be a closed minimal immersion whose induced metric $g$ is rotationally symmetric. Suppose that the Ricci tensor $\operatorname{Ric}_g$ and the sectional curvature $\sec_g$ of $S^n$ satisfy
	\begin{equation}\label{eq:higher-pinching}
		\operatorname{Ric}_g\geq(n-1)\frac{n}{(r+1)(r+n)}g,\quad\sec_g\leq\frac{n}{r(r+n-1)}.
	\end{equation}
	Then
	\begin{equation}\label{eq:higher-endpoints}
		\sec_g\equiv\frac{n}{r(r+n-1)}\quad\text{or}\quad\sec_g\equiv\frac{n}{(r+1)(r+n)}.
	\end{equation}
\end{theorem}

The rotational hypothesis is used only in the spectral comparisons. After scaling by the lower endpoint, the Ricci hypothesis in \eqref{eq:higher-pinching} becomes exactly $\operatorname{Ric}\geq(n-1)g$, so the forward ordered comparison and its rigidity follow directly from Lin, Wang, and Xu \cite[Theorem 6.5]{LinWangXu}. We prove the reverse comparison from the upper sectional-curvature bound by using the radial factorization of Lin, Wang, and Xu \cite[Sections 3-5]{LinWangXu} with the opposite signs in both exact defect identities. Theorem \ref{thm:higher-rotational-gap} then follows by Theorem \ref{thm:reverse-rotational}, Lin, Wang, and Xu \cite[Theorem 6.5]{LinWangXu}, and Takahashi's theorem \cite[Theorem 3]{Takahashi}.

\section{Preliminaries}\label{sec:prelim}

We first fix the notation used throughout the paper. The set of nonnegative integers is $\Nzero=\{0,1,2,\cdots\}$. Unless otherwise stated, the letters $i,j,l,m,q$ denote elements of $\Nzero$, while the Simon-gap index $s$ is a positive integer. The symbol $\Sph^n\subset\R^{n+1}$ denotes the unit round $n$-sphere, and $g_{\Sph^n}$ denotes its round metric of sectional curvature $1$. The notation $g_{\rd}$ is reserved for the unit round metric on $\Sph^2$. In the immersed-surface parts, $N\geq2$ is an integer and
\begin{equation*}
	x\colon M^2\to\Sph^N
\end{equation*}
is a smooth immersion of a connected closed surface, and $g=x^*g_{\Sph^N}$ is its induced metric. The symbols $\nabla$, $R$, $K=K_g$, $\Delta_g=\operatorname{div}_g\nabla$, and $\dd\mu_g$ denote, respectively, the Levi-Civita connection, the Riemann curvature tensor, the Gaussian curvature, the geometric Laplacian, and the Riemannian area measure of $g$. Our sign conventions are
\begin{equation*}
	R(X,Y)=\nabla_X\nabla_Y-\nabla_Y\nabla_X-\nabla_{[X,Y]},\quad K_g=\langle R(e_1,e_2)e_2,e_1\rangle,
\end{equation*}
where $[X,Y]$ is the Lie bracket, $\langle\cdot,\cdot\rangle$ is the inner product induced by $g$, and $e_1,e_2$ is a local $g$-orthonormal frame. We write $\operatorname{div}=\operatorname{div}_g$ when the metric is fixed. Thus $-\Delta_g$ is nonnegative. The second fundamental form of $x$ is denoted by $A$, its pointwise squared norm by $|A|^2$, and its mean curvature vector and its length by
\begin{equation*}
	\mathbf H=\frac12\operatorname{tr}_g A,\qquad H=|\mathbf H|,
\end{equation*}
where $\operatorname{tr}_g$ denotes the trace over the two tangent variables using $g$. The immersion is minimal when $\mathbf H\equiv0$, and it has parallel mean curvature vector when $\nabla^\perp\mathbf H=0$, where $\nabla^\perp$ denotes the induced normal connection on the normal bundle. In the PMC part we use
\begin{equation*}
	\kappa=1+H^2\quad\text{and}\quad A^\circ=A-g\otimes\mathbf H.
\end{equation*}
The immersion is linearly full when its image is not contained in a proper totally geodesic subsphere. The letter $M$ always refers to the original immersed surface, whereas $\Sigma$ denotes a generic oriented smooth $2$-sphere used in the spectral and line-bundle arguments. In those arguments, $g$ denotes an arbitrary smooth Riemannian metric on $\Sigma$, not necessarily one induced by an immersion. A hat, as in $\widehat M$ or $\widehat g$, denotes passage to the oriented double cover or pullback to that cover. The map of the cover is written $\pi$.\par
For a Hermitian complex vector bundle $E\to\Sigma$, $C^\infty(E)$, $L^2(E)$, and $W^{k,2}(E)$ denote, respectively, smooth sections, square-integrable sections, and Sobolev sections with $k$ weak covariant derivatives in $L^2$. The bundle metric and the area measure $\dd\mu_g$ are used in all these spaces. Pointwise Hermitian products and norms are written $\langle\cdot,\cdot\rangle$ and $|\cdot|$; the global $L^2$ product and norm are
\begin{equation*}
	\langle u,v\rangle_2=\int_\Sigma\langle u,v\rangle\,\dd\mu_g,\quad\|u\|_2^2=\langle u,u\rangle_2.
\end{equation*}
The Hermitian products are linear in the first variable. When no confusion is possible, the subscript $2$ on the global inner product is suppressed. For a linear operator $P$, $\Dom(P)$, $\ker P$, $\Ran P$, $P^*$, and $\Spec(P)$ denote its domain, kernel, range, Hilbert-space adjoint, and spectrum. The superscript $^*$ has two other standard uses: for a smooth map $f$, $f^*$ denotes pullback, while for a complex vector bundle $E$, $E^*$ denotes its complex dual bundle; the intended meaning is determined by the type of the object carrying the superscript. The symbol $P^\dagger$ is used only for a formal $L^2$-adjoint differential expression before closure. The identity map on the relevant fiber or Hilbert space is denoted by $\Id$.\par
If $P$ is a lower-bounded self-adjoint operator with compact resolvent, its eigenvalues, repeated according to multiplicity and indexed from zero, are
\begin{equation*}
	\lambda_0(P)\leq\lambda_1(P)\leq\cdots.
\end{equation*}
For the scalar Laplacian we abbreviate $\lambda_j(g)=\lambda_j(-\Delta_g)$. If $p$ is the closed sesquilinear form associated with $P$, we write $p[u]=p[u,u]$, and its Rayleigh quotient is $p[u]/\|u\|_2^2$. The zero-based min-max formula is
\begin{equation}\label{eq:minmax-prelim}
	\lambda_j(P)=\inf_{\substack{F\subset\Dom(p)\\\dim_{\C}F=j+1}}\;\sup_{0\ne u\in F}\frac{p[u]}{\|u\|_2^2}.
\end{equation}
For two forms $p,q$ with the same domain, the notation $q\geq p$ means $q[u]\geq p[u]$ for every vector in that common domain; $q\leq p$ is defined analogously. For a real scalar operator the same min-max formula is taken over real subspaces; Remark \ref{rem:complex-multiplicity} explains why the resulting scalar indices agree with those obtained after complexification.\par
For an oriented Riemannian surface $\Sigma$, the metric-compatible complex structure is denoted by $J$. The complexified tangent and cotangent bundles split into their $\pm i$ eigenspaces,
\begin{equation*}
	T\Sigma_{\C}=T^{1,0}\Sigma\oplus T^{0,1}\Sigma,\quad T^*\Sigma_{\C}=\Lambda^{1,0}T^*\Sigma\oplus\Lambda^{0,1}T^*\Sigma.
\end{equation*}
For a complex vector bundle $E$, $\Omega^{p,q}(\Sigma,E)=C^\infty(\Lambda^{p,q}T^*\Sigma\otimes E)$ denotes the smooth $E$-valued forms of type $(p,q)$. The canonical holomorphic line bundle is $\mathcal K_\Sigma=\Lambda^{1,0}T^*\Sigma$, and $\mathcal K_\Sigma^{-m}=(\mathcal K_\Sigma^*)^{\otimes m}$. For a holomorphic bundle $E$, $\bar\partial_E$ denotes its Dolbeault operator, $\mathscr O(E)$ its sheaf of holomorphic sections, and $H_{\bar\partial}^{p,q}(\Sigma,E)$ its Dolbeault cohomology. We write $H^q(\Sigma,E)$ as shorthand for the sheaf cohomology $H^q(\Sigma,\mathscr O(E))$. The standard holomorphic line bundle of degree $n$ over $\CP^1$ is denoted by $\mathcal O(n)$; by the same symbol we also denote its sheaf of holomorphic sections when it occurs inside $H^q$. These cohomology symbols are distinct from the Sobolev spaces $W^{k,2}(E)$.\par
We first recall several standard facts from submanifold theory and complex geometry.
\begin{lemma}\label{lem:gauss}
	For a minimal surface in the unit sphere,
	\begin{equation*}
		K=1-\frac12|A|^2.
	\end{equation*}
\end{lemma}

\begin{lemma}\label{lem:cover}
	Let $M$ be a connected closed surface with $K>0$. Let $\pi\colon\widehat M\to M$ be the identity when $M$ is orientable and the connected orientation double cover otherwise. Then $\widehat M$ is diffeomorphic to $\Sph^2$. If $g$ and an immersion $x$ are lifted by $\pi$, all local submanifold equations pull back. In particular, curvature, minimality or the PMC condition, and pointwise pinching are preserved.
\end{lemma}
The following lemma is the coordinate-eigenvalue characterization proved by Takahashi \cite[Theorem 3]{Takahashi}.
\begin{lemma}[\sourcecite{Takahashi}{Theorem 3}]\label{lem:Takahashi}
	The immersion $x\colon M^n\to\Sph^N$ is minimal if and only if
	\begin{equation*}
		-\Delta_gx=nx.
	\end{equation*}
	If in addition $M$ is closed, then $n\in\Spec(-\Delta_g)$.
\end{lemma}

\begin{lemma}\label{lem:scaling}
	If $\widehat g=cg$ for a constant $c>0$, then
	\begin{equation*}
		\sec_{\widehat g}=c^{-1}\sec_g,\quad\operatorname{Ric}_{\widehat g}=\operatorname{Ric}_g,\quad\lambda_i(\widehat g)=c^{-1}\lambda_i(g),
	\end{equation*}
	where the Ricci tensors are viewed as symmetric $(0,2)$-tensors. In dimension two the first identity reads $K_{\widehat g}=c^{-1}K_g$.
\end{lemma}

\begin{lemma}\label{lem:round-spectrum}
	On the unit round $2$-sphere, let $$E_l=\{f\in C^\infty(\Sph^2;\R):-\Delta_{g_{\rd}}f=l(l+1)f\}$$ denote the real degree-$l$ spherical-harmonic eigenspace. Then
	\begin{equation*}
		\Spec(-\Delta_{g_{\rd}})=\{l(l+1):l\in\Nzero\},\quad\dim_{\R}E_l=2l+1.
	\end{equation*}
	Consequently, the degree-$l$ cluster occupies exactly
	\begin{equation*}
		\lambda_{l^2}(g_{\rd})=\lambda_{l^2+1}(g_{\rd})=\cdots=\lambda_{(l+1)^2-1}(g_{\rd})=l(l+1).
	\end{equation*}
\end{lemma}
\begin{proof}
	If $P$ is a homogeneous harmonic polynomial of degree $l$ on $\R^3$, the polar-coordinate formula for the Euclidean Laplacian shows that $Y=P|_{\Sph^2}$ satisfies $-\Delta_{g_{\rd}}Y=l(l+1)Y$. The harmonic decomposition of homogeneous polynomials gives
	\begin{equation*}
		\dim_{\R} E_l=\binom{l+2}{2}-\binom l2=2l+1
	\end{equation*}
	and the orthogonal spherical-harmonic decomposition is complete. Since $\sum_{r=0}^{l-1}(2r+1)=l^2$, the first index of the $l$-th cluster is $l^2$, and its last index is $l^2+2l=(l+1)^2-1$.
\end{proof}
We recall the complex-geometric facts used in the kernel computations below. Standard references are \cite{Forster,GriffithsHarris,JostRiemann,Wells}. Let $(\Sigma,g)$ be a connected oriented smooth Riemannian surface. At each point there is a unique endomorphism $J\colon T\Sigma\to T\Sigma$ such that
\begin{equation*}
	J^2=-\Id,\quad g(JX,JY)=g(X,Y),
\end{equation*}
and $(X,JX)$ is positively oriented whenever $X\ne0$. Equivalently, for every local positively oriented orthonormal frame $e_1,e_2$, one has $Je_1=e_2$ and $Je_2=-e_1$. Extend $J$ complex linearly to $T\Sigma_{\C}=T\Sigma\otimes_{\R}\C$ and define
\begin{equation*}
	\Pi^{1,0}=\frac12(\Id-iJ),\quad\Pi^{0,1}=\frac12(\Id+iJ).
\end{equation*}
In real dimension two the almost-complex structure $J$ is automatically integrable. Concretely, every point has an oriented isothermal coordinate chart $(u,v)$ in which
\begin{equation*}
	g=e^{2\varphi}(du^2+dv^2),\quad J\partial_u=\partial_v,
\end{equation*}
for a smooth real function $\varphi$. The complex coordinate $z=u+iv$ is then holomorphic. We use the standard abbreviations $\partial_z, \partial_{\bar z}, dz$ and $d\bar z$. More generally, $d=\partial+\bar\partial$ on complex-valued differential forms, with
\begin{equation*}
	\partial\colon\Omega^{p,q}(\Sigma)\to\Omega^{p+1,q}(\Sigma),\quad\bar\partial\colon\Omega^{p,q}(\Sigma)\to\Omega^{p,q+1}(\Sigma).
\end{equation*}
The canonical line bundle of the Riemann surface is
\begin{equation*}
	\mathcal K_\Sigma=\Lambda^{1,0}T^*\Sigma.
\end{equation*}
Its dual is the holomorphic tangent line bundle, $\mathcal K_\Sigma^*=T^{1,0}\Sigma$. For $m\in\Nzero$ we write $\mathcal K_\Sigma^{-m}=(\mathcal K_\Sigma^*)^{\otimes m}$, with $\mathcal K_\Sigma^0=\Sigma\times\C$. A holomorphic structure on a smooth complex line bundle $E\to\Sigma$ can be described by a first-order operator
\begin{equation*}
	\bar\partial_E\colon\Omega^{0,0}(\Sigma,E)\to\Omega^{0,1}(\Sigma,E)
\end{equation*}
satisfying the Leibniz rule
\begin{equation*}
	\bar\partial_E(f\sigma)=\bar\partial f\otimes\sigma+f\bar\partial_E\sigma
\end{equation*}
for $f\in C^\infty(\Sigma;\C)$ and $\sigma\in C^\infty(E)$, together with $\bar\partial_E^2=0$. A smooth local section $e$ is a holomorphic frame precisely when $e$ is nowhere zero and $\bar\partial_Ee=0$. In such a frame, a section $fe$ is holomorphic exactly when $f$ is holomorphic. The sheaf of local holomorphic sections is denoted by $\mathscr O(E)$. The operator extends to $E$-valued forms by the graded Leibniz rule,
\begin{equation*}
	\bar\partial_E\colon\Omega^{0,q}(\Sigma,E)\to\Omega^{0,q+1}(\Sigma,E).
\end{equation*}
Its cohomology is
\begin{equation*}
	H_{\bar\partial}^{0,q}(\Sigma,E)=\frac{\ker\left\{\bar\partial_E\colon\Omega^{0,q}(E)\to\Omega^{0,q+1}(E)\right\}}{\Ran\left\{\bar\partial_E\colon\Omega^{0,q-1}(E)\to\Omega^{0,q}(E)\right\}}.
\end{equation*}
Because $\Sigma$ has complex dimension one, $\Omega^{0,q}(\Sigma,E)=0$ for $q\ge2$, and the Dolbeault complex reduces to
\begin{equation*}
	0\longrightarrow\Omega^{0,0}(\Sigma,E)\xrightarrow{\ \bar\partial_E\ }\Omega^{0,1}(\Sigma,E)\longrightarrow0.
\end{equation*}
In particular, every $E$-valued $(0,1)$-form is automatically $\bar\partial_E$-closed. We shall use the following standard Dolbeault and Hodge identifications.
\begin{theorem}[\sourcecite{Wells}{Chapters II and IV}]\label{thm:dolbeault-hodge}
	Let $E$ be a Hermitian holomorphic line bundle over a compact Riemann surface.  Then
	\begin{equation*}
		H_{\bar\partial}^{0,q}(\Sigma,E)\simeq H^q(\Sigma,\mathscr O(E)).
	\end{equation*}
	Let $\bar\partial_E^*$ be the Hilbert-space $L^2$ adjoint of $\bar\partial_E$ and
	\begin{equation*}
		\Box_{\bar\partial,E}=\bar\partial_E\bar\partial_E^*+\bar\partial_E^*\bar\partial_E
	\end{equation*}
	be the Dolbeault Laplacian operator. Then every Dolbeault cohomology class has a unique smooth harmonic representative, and
	\begin{equation*}
		H_{\bar\partial}^{0,q}(\Sigma,E)\simeq\ker\bigl(\Box_{\bar\partial,E}|_{\Omega^{0,q}(E)}\bigr).
	\end{equation*}
	For $q=1$, harmonicity is equivalent simply to $\bar\partial_E^*\eta=0$, because $\bar\partial_E\eta=0$ automatically.
\end{theorem}
We apply them only to $E=\mathcal O(n)$ on $\CP^1$. A Hermitian metric on $E$ is a smoothly varying positive Hermitian inner product on the fibers. A connection $D^E$ is unitary when
\begin{equation*}
	X\langle\sigma,\tau\rangle=\langle D_X^E\sigma,\tau\rangle+\langle\sigma,D_X^E\tau\rangle
\end{equation*}
for real vector fields $X$. Its complexification decomposes by type as $$D^E=(D^E)^{1,0}+(D^E)^{0,1}.$$
\begin{theorem}[\sourcecite{Wells}{Chapter III}]\label{thm:chern-general}
	Every Hermitian holomorphic line bundle $(E,h)$ has a unique unitary connection $D^E$ satisfying
	\begin{equation*}
		(D^E)^{0,1}=\bar\partial_E.
	\end{equation*}
	It is called the Chern connection. If $e$ is a local holomorphic frame and $h_e=|e|_h^2$, then
	\begin{equation*}
		D^E(fe)=\left(df+f\partial\log h_e\right)e.
	\end{equation*}
	The Chern connection on a dual or tensor product is the connection induced from the Chern connections of the factors.
\end{theorem}
\begin{lemma}[\sourcecite{Wells}{Chapter III}]\label{lem:chern-connection}
	The connection induced by the Levi-Civita connection on $T^{1,0}\Sigma=\mathcal K_\Sigma^*$ is the Chern connection of the Hermitian holomorphic tangent line bundle.  Consequently, the induced connection $\nabla^{(m)}$ on $\mathcal K_\Sigma^{-m}=(T^{1,0}\Sigma)^{\otimes m}$ satisfies
	\begin{equation*}
		\bar\partial_m=\left(\nabla^{(m)}\right)^{0,1}\colon C^\infty(\mathcal K_\Sigma^{-m})\to\Omega^{0,1}(\Sigma,\mathcal K_\Sigma^{-m}).
	\end{equation*}
\end{lemma}
When $\Sigma$ is topologically a two-sphere, the uniformization theorem \cite{JostRiemann} gives a biholomorphism
\begin{equation*}
	(\Sigma,J)\simeq\CP^1.
\end{equation*}
All dimensions of cohomology groups used below are invariant under biholomorphism. Let
\begin{equation*}
	U_0=\CP^1\setminus\{\infty\},\quad U_\infty=\CP^1\setminus\{0\},
\end{equation*}
with affine coordinates $z$ on $U_0$ and $w=1/z$ on $U_\infty$. For an integer $n$, the holomorphic line bundle $\mathcal O(n)$ is defined by nonvanishing holomorphic frames $e_0$ and $e_\infty$ related on $U_0\cap U_\infty=\C^*$ by
\begin{equation}\label{eq:On-transition}
	e_0=z^{-n}e_\infty.
\end{equation}
With this convention, we have
\begin{equation*}
	\begin{aligned}
		\mathcal O(a)\otimes\mathcal O(b)&\simeq\mathcal O(a+b),\\
		\mathcal O(n)^*&\simeq\mathcal O(-n),\\
		\deg\mathcal O(n)&=n.
	\end{aligned}
\end{equation*}
Since $dz=-w^{-2}dw$ on the overlap, the canonical bundle has transition of degree $-2$ and therefore
\begin{equation*}
	\mathcal K_{\CP^1}\simeq\mathcal O(-2).
\end{equation*}
We shall use the following standard form of Serre duality for holomorphic line bundles over compact Riemann surfaces.
\begin{theorem}[Serre duality \cite{Forster}]\label{thm:serre-duality}
	For a holomorphic line bundle $E$ over a compact Riemann surface,
	\begin{equation*}
		H^1(\Sigma,E)\simeq H^0(\Sigma,\mathcal K_\Sigma\otimes E^*)^*.
	\end{equation*}
	In particular,
	\begin{equation}\label{eq:serre-CP1}
		H^1(\CP^1,\mathcal O(n))\simeq H^0(\CP^1,\mathcal O(-n-2))^*.
	\end{equation}
\end{theorem}
We shall use the following standard cohomology formulas for holomorphic line bundles on $\CP^1$, which follow from the Riemann-Roch theorem and Serre duality; see, for example, \cite[Chapter~I]{GriffithsHarris}.
\begin{lemma}[Cohomology on $\CP^1$]\label{lem:CP1-cohomology}
	For every integer $n\geq0$,
	\begin{equation*}
		\dim_\C H^0(\CP^1,\mathcal O(n))=n+1,\qquad H^1(\CP^1,\mathcal O(n))=0.
	\end{equation*}
	Moreover, the first group consists of polynomials of degree at most $n$ in the affine coordinate $z$, interpreted as sections using \eqref{eq:On-transition}.
\end{lemma}
\begin{proof}
	A global holomorphic section is represented by $P(z)e_0$ on $U_0$ and $Q(w)e_\infty$ on $U_\infty$. The transition rule \eqref{eq:On-transition} gives $Q(w)=w^nP(1/w)$, so holomorphy at $w=0$ is equivalent to $P$ being a polynomial of degree at most $n$. This proves the formula for $H^0$. By Serre duality \eqref{eq:serre-CP1},
	\begin{equation*}
		H^1(\CP^1,\mathcal O(n))\simeq H^0(\CP^1,\mathcal O(-n-2))^*.
	\end{equation*}
	A holomorphic line bundle of negative degree on $\CP^1$ has no nonzero holomorphic section, so the latter space vanishes.
\end{proof}

\section{The line-bundle construction and the Dolbeault realization}\label{sec:bundle}

The line-bundle construction used in this section originates in Lin, Wang, and Xu \cite[Sections 7-10]{LinWangXu}. In particular, the bundles and first-order operators are introduced in \cite[Definitions 7.1, 7.4, and 7.5]{LinWangXu}, and the curvature commutator and the resulting spectral ladder are developed there. We give a self-contained presentation because the argument below requires the operators and their identities on a common closed form domain. We also make explicit several standard functional-analytic details concerning the closed realizations and their domains. The additional use made in the present paper is that the same exact form identity can be read in the opposite direction when $K_g\leq1$; this reverse comparison is then applied to the Simon gaps.\par
Let $(\Sigma,g)$ be an oriented smooth Riemannian $2$-sphere, equipped with the complex structure and type decomposition of Section \ref{sec:prelim}. For $m\in\Nzero$ define the Hermitian holomorphic line bundle $L_m$ and its $L^2$ Hilbert space $H_m$ by
\begin{equation*}
	\begin{aligned}
		L_m&=(T^{1,0}\Sigma)^{\otimes m}=\mathcal K_\Sigma^{-m},\\
		L_0&=\Sigma\times\C,\\
		H_m&=L^2(L_m).
	\end{aligned}
\end{equation*}
The exponent $\otimes m$ denotes the $m$-fold complex tensor power, with the zeroth tensor power interpreted as the trivial complex line bundle. The equality with $\mathcal K_\Sigma^{-m}$ is an equality of Hermitian holomorphic line bundles by Lemma \ref{lem:chern-connection}. The Levi-Civita connection preserves $g$ and $J$, and hence induces a unitary connection, denoted $\nabla^{(m)}$, on every $L_m$.\par
Choose a local oriented orthonormal frame $e_1,e_2$ and put
\begin{equation*}
	Z=\frac{e_1-ie_2}{\sqrt2}.
\end{equation*}
Then $Z$ is a local unitary frame of $T^{1,0}\Sigma$, meaning $\langle Z,Z\rangle=1$. Let $d\mu_g=\theta^1\wedge\theta^2$ denote the positively oriented area form, where $\theta^1,\theta^2$ is the coframe dual to $e_1,e_2$. Let the real Levi-Civita connection local $1$-form $\omega\in\Omega^1(\Sigma;\R)$ be defined by
\begin{equation}\label{eq:connection}
	\nabla e_1=\omega\otimes e_2,\quad\nabla e_2=-\omega\otimes e_1.
\end{equation}
Here $\Omega^1(\Sigma;\R)$ is the space of smooth real $1$-forms. Consequently $\nabla Z=i\omega\otimes Z$ and
\begin{equation}\label{eq:curvature-form}
	d\omega=-K_g\,d\mu_g,
\end{equation}
where $d$ is the exterior derivative. Indeed,
\begin{equation*}
	R(e_1,e_2)e_1=d\omega(e_1,e_2)e_2,\quad\langle R(e_1,e_2)e_1,e_2\rangle=-K_g,
\end{equation*}
where the second identity uses the curvature convention of Section \ref{sec:prelim}. For a vector field $X$ and the scalar coefficient $u$ of a local section $uZ^{\otimes m}$, define the scalar covariant derivative $D_X^{(m)}u$ by
\begin{equation*}
	\nabla_X^{(m)}(uZ^{\otimes m})=D_X^{(m)}uZ^{\otimes m},\quad D_X^{(m)}u=X(u)+im\omega(X)u.
\end{equation*}
Thus $D_X^{(m)}$ is the local coefficient expression of the induced connection $\nabla^{(m)}$.\par
We repeat the verification of global well-definedness and the adjoint formula in the sign conventions used here.
\begin{lemma}[\sourcecite{LinWangXu}{Definitions 7.4 and 7.5}]\label{lem:global-B}
	The rule
	\begin{equation}\label{eq:Bm-local}
		B_m(uZ^{\otimes m})=\bigl(D_{e_1}^{(m)}u+iD_{e_2}^{(m)}u\bigr)Z^{\otimes(m+1)}
	\end{equation}
	defines a global first-order differential operator $C^\infty(L_m)\to C^\infty(L_{m+1})$. Its formal $L^2$ adjoint, denoted $B_m^\dagger$ at this stage, is
	\begin{equation}\label{eq:Bmstar-local}
		B_m^\dagger(vZ^{\otimes(m+1)})=-\left(D_{e_1}^{(m+1)}v-iD_{e_2}^{(m+1)}v\right)Z^{\otimes m}.
	\end{equation}
\end{lemma}
\begin{proof}
	Rotate the oriented orthonormal frame through a smooth angle $\alpha$, and denote the quantities in the new frame by primes. Then we have
	\begin{align*}
		e_1'&=e_1\cos\alpha+e_2\sin\alpha,\\
		e_2'&=-e_1\sin\alpha+e_2\cos\alpha,\\
		Z'&=e^{i\alpha}Z,\quad\text{and}\\
		\omega'&=\omega+d\alpha.
	\end{align*}
	If $uZ^{\otimes m}=u'(Z')^{\otimes m}$, then $u'=e^{-im\alpha}u$. Direct substitution gives
	\begin{equation*}
		D_{e_1'}^{\prime(m)}u'+iD_{e_2'}^{\prime(m)}u'=e^{-i(m+1)\alpha}\left(D_{e_1}^{(m)}u+iD_{e_2}^{(m)}u\right).
	\end{equation*}
	The factor cancels the transformation of $(Z')^{\otimes(m+1)}$, proving global well-definedness. Write $\omega_j=\omega(e_j)$. From \eqref{eq:connection},
	\begin{equation}\label{eq:div-frame}
		\operatorname{div}e_1=\omega_2,\qquad\operatorname{div}e_2=-\omega_1.
	\end{equation}
	Regarding $e_j$ as the scalar first-order differential operator $f\mapsto e_j(f)$, integration by parts and $(e_j)^\dagger=-e_j-\operatorname{div}e_j$ give
	\begin{equation*}
		\left(D_{e_1}^{(m)}+iD_{e_2}^{(m)}\right)^\dagger=-\left(D_{e_1}^{(m+1)}-iD_{e_2}^{(m+1)}\right),
	\end{equation*}
	which is \eqref{eq:Bmstar-local}.
\end{proof}

We state the closed-realization facts separately and include the elliptic argument, since the precise domains will be used in passing from the smooth identities to closed quadratic forms.
\begin{lemma}[\sourcecite{LinWangXu}{Definitions 7.4 and 7.5}]\label{lem:closed-realization}
	Initially regard \eqref{eq:Bm-local} as the densely defined operator
	\begin{equation*}
		B_m^{(0)}\colon C^\infty(L_m)\subset H_m\to H_{m+1}.
	\end{equation*}
	It is closable. Its graph closure is denoted again by $B_m$, and its Hilbert-space adjoint is denoted by $B_m^*$. They obey
	\begin{equation*}
		\Dom(B_m)=W^{1,2}(L_m),\quad
		\Dom(B_m^*)=W^{1,2}(L_{m+1}).
	\end{equation*}
	Their graph norms are equivalent to the indicated $W^{1,2}$ norms, and on smooth sections the Hilbert adjoint $B_m^*$ agrees with the differential expression $B_m^\dagger$ in \eqref{eq:Bmstar-local}.
\end{lemma}
\begin{proof}
	The formal adjoint in Lemma \ref{lem:global-B} is defined on the dense space $C^\infty(L_{m+1})$, so $B_m^{(0)}$ is closable. For $p\in\Sigma$ and $\xi\in T_p^*\Sigma$, its principal symbol, up to the factor $i$, is multiplication by
	\begin{equation*}
		\xi(e_1)+i\xi(e_2),\quad|\xi(e_1)+i\xi(e_2)|^2=|\xi|_g^2.
	\end{equation*}
	Thus $B_m^{(0)}$ is first-order elliptic. The global elliptic estimate on the closed surface, together with boundedness of its coefficients, gives constants $\gamma_m,\Gamma_m>0$ such that
	\begin{equation*}
		\gamma_m\|u\|_{W^{1,2}}\le\left(\|B_m^{(0)}u\|_2^2+\|u\|_2^2\right)^{1/2}\le\Gamma_m\|u\|_{W^{1,2}}
	\end{equation*}
	for every smooth section. Since $C^\infty(L_m)$ is dense in $W^{1,2}(L_m)$, the completion in the graph norm is exactly $W^{1,2}(L_m)$; this proves the first domain identity. If $v\in\Dom(B_m^*)$, then the distribution $B_m^\dagger v\in L^2(L_m)$.  First-order elliptic regularity \cite{TaylorPDE}, applied in local trivializations, gives $v\in W^{1,2}(L_{m+1})$.  Conversely, if $v\in W^{1,2}(L_{m+1})$, approximate $v$ in $W^{1,2}$ by smooth sections and use the formal integration-by-parts identity to see that $u\mapsto\langle B_m u,v\rangle_2$ is $L^2$-bounded on $\Dom(B_m)$.  Hence $v\in\Dom(B_m^*)$, and $B_m^*v=B_m^\dagger v$ in the sense of distribution. The elliptic estimate for $B_m^\dagger$ gives the adjoint graph-norm equivalence.
\end{proof}

We identify the first-order ladder operator with the Dolbeault operator of $L_m=\mathcal K_\Sigma^{-m}$.  Its kernel and adjoint kernel are then holomorphic invariants of the genus-zero Riemann surface and do not depend on the sign of $K_g-1$. Choose a local oriented orthonormal frame $e_1,e_2$ with dual coframe $\theta^1,\theta^2$, and set
\begin{equation*}
	\zeta=\frac{\theta^1+i\theta^2}{\sqrt2},\qquad\bar\zeta=\frac{\theta^1-i\theta^2}{\sqrt2}.
\end{equation*}
Then $\zeta$ and $\bar\zeta$ are unit covectors of type $(1,0)$ and $(0,1)$, respectively, and $\zeta(Z)=1$ for the unit vector $Z=(e_1-ie_2)/\sqrt2$ introduced in Section \ref{sec:bundle}.\par
The Dolbeault realization of the line-bundle operator is developed in \cite[Section 10]{LinWangXu}. The following lemma records that identification in our notation and normalizations.
\begin{lemma}[\sourcecite{LinWangXu}{Section 10}]\label{lem:dolbeault-identification}
	The rule
	\begin{equation}\label{eq:Im}
		I_m(\bar\zeta\otimes Z^{\otimes m})=Z^{\otimes(m+1)}
	\end{equation}
	defines a global unitary bundle isomorphism
	\begin{equation*}
		I_m\colon\Lambda^{0,1}T^*\Sigma\otimes L_m\to L_{m+1}.
	\end{equation*}
	If $\bar\partial_m$ is the Dolbeault operator of $L_m$, then
	\begin{equation}\label{eq:Bm-Dolbeault}
		B_m=\sqrt{2}I_m\bar\partial_m.
	\end{equation}
\end{lemma}
\begin{proof}
	Under the frame rotation used above, $Z'=e^{i\alpha}Z$ and $\bar\zeta'=e^{i\alpha}\bar\zeta$. Both sides of \eqref{eq:Im} acquire the same factor $e^{i(m+1)\alpha}$, so $I_m$ is global; the displayed frames are unitary. By Lemma \ref{lem:chern-connection}, the $(0,1)$ part of the induced connection is the Dolbeault operator, and therefore
	\begin{equation*}
		\bar\partial_m(uZ^{\otimes m})=\frac1{\sqrt2}\left(D_{e_1}^{(m)}u+iD_{e_2}^{(m)}u\right)\bar\zeta\otimes Z^{\otimes m}.
	\end{equation*}
	Applying $\sqrt2 I_m$ gives \eqref{eq:Bm-local}.
\end{proof}

We reproduce the curvature-independent kernel calculation using Lemma \ref{lem:CP1-cohomology}, so that both the kernel dimension and the vanishing of the adjoint kernel follow directly from the cohomology of $\mathcal{O}(2m)$.
\begin{lemma}[\sourcecite{LinWangXu}{Lemmas 8.2, 8.3, and 10.2}]\label{lem:kernels}
	For every $m\geq0$ and every smooth metric on $\Sph^2$,
	\begin{equation*}
		\dim_\C\ker B_m=2m+1,\quad\ker B_m^*=\{0\}.
	\end{equation*}
	These dimensions are independent of the curvature of $g$.
\end{lemma}
\begin{proof}
	By Lemma \ref{lem:CP1-cohomology}, we have $L_m=\mathcal K_\Sigma^{-m}\simeq\mathcal O(2m)$. Every weak $L^2$ section in $\ker B_m$ is smooth by elliptic regularity. Lemma \ref{lem:dolbeault-identification} therefore gives
	\begin{equation*}
		\ker B_m=H^0(\CP^1,\mathcal O(2m)),
	\end{equation*}
	whose complex dimension is $2m+1$. Now let $\Psi\in\Dom(B_m^*)$ satisfy $B_m^*\Psi=0$ and set $\eta=I_m^{-1}\Psi$. We denote by $\bar\partial_m^*$ the Hilbert-space $L^2$ adjoint of $\bar\partial_m$. Since $I_m$ is a smooth unitary bundle map, \eqref{eq:Bm-Dolbeault} implies, first on smooth sections and then for the Hilbert adjoints,
	\begin{equation*}
		B_m^*=\sqrt2\,\bar\partial_m^*I_m^{-1}.
	\end{equation*}
	First-order elliptic regularity for $B_m^*$ implies that $\Psi$, and hence $\eta$, is smooth. Therefore $\bar\partial_m^*\eta=0$. Because $\Sigma$ has complex dimension one, $\Lambda^{0,2}T^*\Sigma=0$, so $\bar\partial_m\eta=0$ automatically. Thus $\eta$ is a Dolbeault-harmonic $(0,1)$-form. By Theorem \ref{thm:dolbeault-hodge}, its harmonic space is canonically isomorphic to Dolbeault and sheaf cohomology:
	\begin{equation*}
		\eta\in\ker\Box_{\bar\partial,L_m}\simeq H_{\bar\partial}^{0,1}(\Sigma,L_m)\simeq H^1(\CP^1,\mathcal O(2m))=0.
	\end{equation*}
	The last equality is Lemma \ref{lem:CP1-cohomology}, and hence $\eta=0$. Alternatively, using the definition of the zero cohomology class, there is a smooth section $\sigma\in C^\infty(L_m)$ with $\eta=\bar\partial_m\sigma$. Since $\Sigma$ is closed, integration by parts gives
	\begin{equation*}
		\|\eta\|_2^2=\langle\eta,\bar\partial_m\sigma\rangle_2=\langle\bar\partial_m^*\eta,\sigma\rangle_2=0.
	\end{equation*}
	Consequently $\eta=0$, and therefore $\Psi=0$.
\end{proof}

We include the local calculation in order to verify the signs and normalizations used in the subsequent closed-form identity.

\begin{lemma}[\sourcecite{LinWangXu}{Lemma 7.7}]\label{lem:commutator}
	On smooth sections,
	\begin{equation}\label{eq:B0-lap}
		B_0^*B_0=-\Delta_g,
	\end{equation}
	and for every $m\geq0$,
	\begin{equation}\label{eq:operator-ladder}
		B_mB_m^*-B_{m+1}^*B_{m+1}=2(m+1)K_g\quad\text{on }C^\infty(L_{m+1}).
	\end{equation}
\end{lemma}

\begin{proof}
	For a function $u$, we have $B_0u=(e_1u+ie_2u)Z$,
	so Lemma \ref{lem:global-B} gives
	\begin{align*}
		B_0^*B_0u&=-\left(D_{e_1}^{(1)}-iD_{e_2}^{(1)}\right)(e_1u+ie_2u)\\
		&=-\left(e_1-ie_2+i\omega_1+\omega_2\right)(e_1u+ie_2u).
	\end{align*}
	Torsion-freeness gives that $[e_1,e_2]=-\omega_1e_1-\omega_2e_2$. Therefore
	\begin{align*}
		\Delta_gu&=e_1e_1u+e_2e_2u+\omega_2e_1u-\omega_1e_2u\\
		&=(e_1-ie_2+i\omega_1+\omega_2)(e_1u+ie_2u),
	\end{align*}
	where the last equality uses \eqref{eq:div-frame}. This proves \eqref{eq:B0-lap} with the stated sign convention. For the commutator, first take $m\geq1$, fix $p\in\Sigma$, and choose the frame geodesic at $p$, meaning $\nabla_{e_i}e_j(p)=0$ for $i,j\in\{1,2\}$. Then $\omega_1(p)=\omega_2(p)=0$, $[e_1,e_2](p)=0$, and
	\begin{equation}\label{eq:domega-point}
		e_1\omega_2(p)-e_2\omega_1(p)=-K_g(p)
	\end{equation}
	by \eqref{eq:curvature-form}. From the two first-order formulas,
	\begin{align*}
		B_m^*B_m(uZ^{\otimes m})&=-\left(D_{e_1}^{(m+1)}-iD_{e_2}^{(m+1)}\right)\left(D_{e_1}^{(m)}+iD_{e_2}^{(m)}\right)u\,Z^{\otimes m},\\
		B_{m-1}B_{m-1}^*(uZ^{\otimes m})&=-\left(D_{e_1}^{(m-1)}+iD_{e_2}^{(m-1)}\right)\left(D_{e_1}^{(m)}-iD_{e_2}^{(m)}\right)u\,Z^{\otimes m}.
	\end{align*}
	Expanding at $p$ and using \eqref{eq:domega-point} gives
	\begin{align*}
		B_m^*B_m(uZ^{\otimes m})&=\left[-e_1e_1u-e_2e_2u-im(e_1\omega_1+e_2\omega_2)u-mK_gu\right]Z^{\otimes m},\\
		B_{m-1}B_{m-1}^*(uZ^{\otimes m})&=\left[-e_1e_1u-e_2e_2u-im(e_1\omega_1+e_2\omega_2)u+mK_gu\right]Z^{\otimes m}.
	\end{align*}
	Subtracting gives $2mK_g$ on $L_m$. Replacing $m$ by $m+1$ proves \eqref{eq:operator-ladder} at the arbitrary point $p$.
\end{proof}

For each $m\in\Nzero$, set
\begin{equation*}
	\alpha_m=m(m+1),
\end{equation*}
and define the sesquilinear forms $a_m$ on $H_m=L^2(L_m)$ and $c_m$ on $H_{m+1}=L^2(L_{m+1})$ by
\begin{align*}
	a_m[u,v]&=\langle B_m u,B_m v\rangle_{2}+\alpha_m\langle u,v\rangle_{2},\quad\Dom(a_m)=W^{1,2}(L_m),\\
	c_m[u,v]&=\langle B_m^*u,B_m^*v\rangle_{2}+\alpha_m\langle u,v\rangle_{2},\quad\Dom(c_m)=W^{1,2}(L_{m+1}).
\end{align*}
Their diagonal values are denoted by $a_m[u]=a_m[u,u]$ and $c_m[u]=c_m[u,u]$. We record the closed realizations, operator domains, and compact-resolvent properties explicitly, since the comparison below will be carried out at the level of closed forms.
\begin{lemma}[\sourcecite{LinWangXu}{Section 9}]\label{lem:closed-forms}
	The forms $a_m,c_m$ are densely defined, nonnegative, and closed. Their self-adjoint operators are
	\begin{equation*}
		A_m=B_m^*B_m+\alpha_m\Id_{H_m},\qquad C_m=B_mB_m^*+\alpha_m\Id_{H_{m+1}},
	\end{equation*}
	with product domains
	\begin{align*}
		\Dom(A_m)&=\{u\in\Dom(B_m):B_m u\in\Dom(B_m^*)\},\\
		\Dom(C_m)&=\{v\in\Dom(B_m^*):B_m^*v\in\Dom(B_m)\}.
	\end{align*}
	Both operators have compact resolvent. In fact, $$\Dom(A_m)=W^{2,2}(L_m),\quad\Dom(C_m)=W^{2,2}(L_{m+1}),$$ and $A_0=-\Delta_g$.
\end{lemma}

\begin{proof}
	After adding the $L^2$ norm, the form norm of $a_m$ is the graph norm of $B_m$, and the form norm of $c_m$ is the graph norm of $B_m^*$. By Lemma \ref{lem:closed-realization}, both are equivalent to the appropriate $W^{1,2}$ norm and are complete. Thus the forms are closed and densely defined. Kato's first representation theorem for closed semibounded forms \cite[Chapter~VI, Section~2]{Kato} gives the associated self-adjoint operators. To identify the first one, note that $u\in\Dom(A_m)$ exactly when there is $f\in H_m$ such that
	\begin{equation*}
		\langle B_m u,B_m\phi\rangle_2+\alpha_m\langle u,\phi\rangle_2=\langle f,\phi\rangle_2\qquad\text{for every }\phi\in W^{1,2}(L_m).
	\end{equation*}
	This is equivalent to $B_m u\in\Dom(B_m^*)$ and $f=B_m^*B_m u+\alpha_m u$. Hence
	\begin{equation*}
		\Dom(A_m)=\{u\in\Dom(B_m):B_m u\in\Dom(B_m^*)\}.
	\end{equation*}
	The argument for $C_m$ is identical. For every $r\in\Nzero$, Rellich compactness \cite{TaylorPDE} of $W^{1,2}(L_r)\hookrightarrow L^2(L_r)$ makes the relevant form-domain embedding compact; therefore $A_m$ and $C_m$ have compact resolvent. On smooth sections the principal symbols of $B_m^*B_m$ and $B_mB_m^*$ are $|\xi|_g^2$ times the identity on their respective line-bundle fibers, so both differential expressions are strongly elliptic of order two. If $u$ belongs to either product domain, the corresponding second-order expression is in $L^2$ distributionally, and elliptic regularity \cite{TaylorPDE} gives $u\in W^{2,2}$. Conversely, a $W^{2,2}$ section lies in $W^{1,2}$, its first-order image lies in $W^{1,2}$, and hence it belongs to the relevant product domain. Thus the operator domains are exactly the displayed $W^{2,2}$ spaces. Iterated elliptic regularity makes all eigenfunctions smooth. Finally, \eqref{eq:B0-lap} and $\alpha_0=0$ give $A_0=-\Delta_g$.
\end{proof}

We state the following identity on the common closed form domain and spell out the density argument leading to that formulation. The fact that this is an exact identity, rather than merely a one-sided estimate, allows it to be used for either sign of $K_g-1$.
\begin{proposition}[\sourcecite{LinWangXu}{(9.1)}]\label{prop:closed-defect}
	For every $m\geq0$ and $v\in W^{1,2}(L_{m+1})$,
	\begin{equation}\label{eq:closed-defect}
		c_m[v]-a_{m+1}[v]=2(m+1)\int_\Sigma(K_g-1)|v|^2\,d\mu_g.
	\end{equation}
\end{proposition}

\begin{proof}
	Pair \eqref{eq:operator-ladder} with a smooth section $v$ and use $\alpha_{m+1}-\alpha_m=2(m+1)$. This proves \eqref{eq:closed-defect} on $C^\infty(L_{m+1})$.  The two forms have the common domain $W^{1,2}(L_{m+1})$, smooth sections are dense there, and multiplication by the smooth function $K_g-1$ is bounded on $L^2$. Passing to a $W^{1,2}$-approximating sequence proves the identity on the full closed form domain.
\end{proof}

The following Hilbert-space pairing is the abstract partner-spectrum principle used in the line-bundle ladder; compare \cite[Lemma 2.2]{LinWangXu}. We include the short proof in order to keep track of the precise zero-based index shift produced by the two kernels.
\begin{lemma}[\sourcecite{LinWangXu}{Lemma 2.2}]\label{lem:partner}
	Let $\mathscr H$ and $\widehat{\mathscr H}$ be complex Hilbert spaces, and let $B\colon\Dom(B)\subset\mathscr H\to\widehat{\mathscr H}$ be a closed densely defined operator.  For every $\mu>0$, the map $B$ is an isomorphism
	\begin{equation*}
		B\colon\ker(B^*B-\mu\Id_{\mathscr H})\to\ker(BB^*-\mu\Id_{\widehat{\mathscr H}}),
	\end{equation*}
	with inverse $\mu^{-1}B^*$. Consequently, for the ladder operators,
	\begin{equation}\label{eq:partner-shift}
		\lambda_{j+2m+1}(A_m)=\lambda_j(C_m),\qquad j,m\geq0.
	\end{equation}
	All eigenvalues in this section are indexed from zero and counted with complex multiplicity.
\end{lemma}

\begin{proof}
	If $u\in\Dom(B^*B)$ and $B^*Bu=\mu u$, then $Bu\ne0$, $Bu\in\Dom(B^*)$, and $B^*(Bu)=\mu u\in\Dom(B)$. Thus $Bu\in\Dom(BB^*)$ and $$BB^*(Bu)=\mu Bu.$$ Conversely, if $v\in\Dom(BB^*)$ and $BB^*v=\mu v$, then $B^*v\ne0$, $B^*v\in\Dom(B)$, and $B(B^*v)=\mu v\in\Dom(B^*)$; hence $B^*v\in\Dom(B^*B)$ and $$B^*B(B^*v)=\mu B^*v.$$ These maps are inverse after multiplication by $\mu^{-1}$. Thus the positive spectra of $B_m^*B_m$ and $B_mB_m^*$ agree with multiplicity. The $\alpha_m$-eigenspace of $A_m$ is $\ker B_m$, of dimension $2m+1$, whereas the $\alpha_m$-eigenspace of $C_m$ is $\ker B_m^*=0$ by Lemma \ref{lem:kernels}. Removing the unpaired bottom block from $A_m$ leaves exactly the ordered spectrum of $C_m$, proving \eqref{eq:partner-shift}.
\end{proof}

The strict min-max argument in the first case is the one used by Lin, Wang, and Xu in the rigidity argument of \cite[Section 9]{LinWangXu}. We state it in an abstract form and include both signs, since the second formulation will be used for the reverse comparison when $K_g\leq1$.
\begin{lemma}[\sourcecite{LinWangXu}{Lemma 9.3}]\label{lem:strict-minmax}
	Let $(\mathcal X,g_{\mathcal X})$ be a connected closed Riemannian manifold with volume measure $\dd\mu_{\mathcal X}$, and let $E\to\mathcal X$ be a Hermitian line bundle. Let $P$ and $Q$ be lower-bounded self-adjoint second-order elliptic operators on $L^2(E)$ with smooth coefficients and compact resolvent, and let $p$ and $q$ denote their associated closed sesquilinear forms. Assume $\Dom(p)=\Dom(q)$. Let $V$ be a smooth real-valued function satisfying $V\geq0$ and $V\not\equiv0$.
	\begin{enumerate}
		\item\label{3.91} If $q[u]=p[u]+\int_{\mathcal X}V|u|^2\dd\mu_{\mathcal X}$, then $\lambda_j(Q)>\lambda_j(P)$ for every $j\geq0$.
		\item\label{3.92} If $q[u]=p[u]-\int_{\mathcal X}V|u|^2\dd\mu_{\mathcal X}$, then $\lambda_j(Q)<\lambda_j(P)$ for every $j\geq0$.
	\end{enumerate}
	All eigenvalues in this section are indexed from zero and counted with complex multiplicity.
\end{lemma}
\begin{proof}
	We use the min-max formula \eqref{eq:minmax-prelim}. For \eqref{3.91}, let $F_j^Q$ be the sum of all eigenspaces of $Q$ whose eigenvalues do not exceed $\lambda_j(Q)$.  It is finite-dimensional, has dimension at least $j+1$, and consists of smooth sections. We claim
	\begin{equation}\label{eq:positive-defect}
		\int_{\mathcal X}V|u|^2\dd\mu_{\mathcal X}>0\quad(0\ne u\in F_j^Q).
	\end{equation}
	Indeed, if the integral vanished, then the smooth section $u$ would vanish on the nonempty open set $\Omega=\{x\in\mathcal X:V(x)>0\}$. Let $r$ be the number of distinct eigenvalues of $Q$ occurring in the spectral decomposition of $u$, and write $u=u_1+\cdots+u_r$, where $Qu_a=\mu_au_a$ and the numbers $\mu_1,\ldots,\mu_r$ are pairwise distinct.  Because a smooth section that vanishes on $\Omega$ has all of its derivatives there equal to zero, $Q^k u=0$ on $\Omega$ for $k=0,\ldots,r-1$.  Hence
	\begin{equation*}
		\sum_{a=1}^r\mu_a^k u_a=0\text{ on }\Omega,\quad k=0,\cdots,r-1.
	\end{equation*}
	The $r\times r$ Vandermonde matrix $(\mu_a^k)_{0\leq k\leq r-1,\,1\leq a\leq r}$ is invertible, so each $u_a$ vanishes on $\Omega$. In a local trivialization, the equation $(Q-\mu_a)u_a=0$ is a scalar second-order elliptic equation with smooth complex coefficients; the standard unique-continuation theorem \cite{Aronszajn} therefore gives $u_a\equiv0$.  This contradicts $u\ne0$ and proves \eqref{eq:positive-defect}. The map $u\mapsto\int_{\mathcal X}V|u|^2\dd\mu_{\mathcal X}$ is continuous on the finite-dimensional space $F_j^Q$.  Compactness of its $L^2$ unit sphere therefore supplies $\delta_j>0$ such that
	\begin{equation*}
		\int_{\mathcal X}V|u|^2\dd\mu_{\mathcal X}\geq\delta_j\|u\|_2^2\quad(u\in F_j^Q).
	\end{equation*}
	Choose any $(j+1)$-dimensional subspace $G_j^Q\subset F_j^Q$. Since the Rayleigh quotient of $Q$ is at most $\lambda_j(Q)$ on $F_j^Q$, min-max gives
	\begin{equation*}
		\lambda_j(P)\leq\sup_{0\ne u\in G_j^Q}\frac{p[u]}{\|u\|_2^2}\leq\lambda_j(Q)-\delta_j.
	\end{equation*}
	For \eqref{3.92}, let $F_j^P$ be the sum of the eigenspaces of $P$ with eigenvalues at most $\lambda_j(P)$ and choose a $(j+1)$-dimensional subspace $G_j^P\subset F_j^P$. The same Vandermonde and unique-continuation argument gives a positive constant $\delta_j$, and
	\begin{equation*}
		\lambda_j(Q)\le\sup_{0\ne u\in G_j^P}\frac{q[u]}{\|u\|_2^2}\le\lambda_j(P)-\delta_j.
	\end{equation*}
	This proves both assertions.
\end{proof}

\begin{remark}\label{rem:complex-multiplicity}
	Although the bundle argument is complex, $A_0=-\Delta_g$ is the complexification of the real scalar Laplacian. If $$E_\lambda^\R=\{f\in C^\infty(\Sigma;\R):-\Delta_gf=\lambda f\}$$ is a real eigenspace, its complexification is $E_\lambda^\R\otimes_\R\C$, whose complex dimension equals $\dim_\R E_\lambda^\R$. Thus the ordered indices obtained below are exactly the usual real scalar indices.
\end{remark}

For $K_g\geq1$, the first comparison below, together with its rigidity statement, is due to Lin, Wang, and Xu \cite[Theorem 1.8]{LinWangXu}. Our proof is a self-contained presentation of their line-bundle argument in the notation established above. The comparison for $K_g\leq1$ is not asserted in their theorem; it follows here by applying the same exact form identity with the opposite order. We state the two directions together in order to make their common mechanism explicit.
\begin{theorem}[\sourcecite{LinWangXu}{Theorem 1.8} for \eqref{firstcom}]\label{thm:comparison-expanded}
	Let $g$ be a smooth metric on $S^2$.  For $l\geq1$ and $0\leq q\leq2l$, we have
	\begin{enumerate}
		\item\label{firstcom} if $K_g\geq1$, then $\lambda_{l^2+q}(g)\geq l(l+1)$;
		\item\label{secondcom} if $K_g\leq1$, then $\lambda_{l^2+q}(g)\leq l(l+1)$.
	\end{enumerate}
	Moreover, if equality holds in either situation for some $i_0\ge1$, then $(S^2,g)$ is isometric to $\Sph^2$.
\end{theorem}

\begin{proof}
	If $K_g\geq1$, \eqref{eq:closed-defect} gives $c_m\geq a_{m+1}$. Min-max and \eqref{eq:partner-shift} yield
	\begin{equation}\label{eq:forward-rung}
		\lambda_{j+2m+1}(A_m)=\lambda_j(C_m)\geq\lambda_j(A_{m+1}).
	\end{equation}
	If $K_g\leq1$, the same exact identity gives $c_m\leq a_{m+1}$ and
	\begin{equation}\label{eq:reverse-rung}
		\lambda_{j+2m+1}(A_m)=\lambda_j(C_m)\leq\lambda_j(A_{m+1}).
	\end{equation}
	When $K_g\not\equiv1$, the functions $2(m+1)(K_g-1)$ in the first case and $2(m+1)(1-K_g)$ in the second are nonnegative and positive on a nonempty open set. Lemma \ref{lem:strict-minmax} makes every applicable rung strict. Fix $l,q$ in the stated range and define
	\begin{equation*}
		j_m=l^2-m^2+q,\quad 0\leq m\leq l.
	\end{equation*}
	Then
	\begin{equation*}
		j_m-j_{m+1}=2m+1,\quad j_{m+1}\geq0.
	\end{equation*}
	Applying \eqref{eq:forward-rung} or \eqref{eq:reverse-rung} at the index $j_{m+1}$ gives
	\begin{equation*}
		\lambda_{j_m}(A_m)=\lambda_{j_{m+1}}(C_m)\mathrel{\gtreqless}\lambda_{j_{m+1}}(A_{m+1}),
	\end{equation*}
	where $\mathrel{\gtreqless}$ means $\geq$ in the case $K_g\geq1$ and $\leq$ in the case $K_g\leq1$. Iteration from $m=0$ through $m=l-1$ gives
	\begin{equation}\label{eq:index-chain}
		\lambda_{l^2+q}(A_0)\mathrel{\gtreqless}\lambda_q(A_l).
	\end{equation}
	Since $B_l^*B_l\geq0$ and $\dim_\C\ker B_l=2l+1$, the indices $q=0,\ldots,2l$ are precisely its zero modes. Hence
	\begin{equation}\label{eq:terminal}
		\lambda_q(A_l)=\alpha_l=l(l+1),\quad0\leq q\leq2l.
	\end{equation}
	Because $A_0=-\Delta_g$, \eqref{eq:index-chain} and \eqref{eq:terminal} prove the theorem.
\end{proof}

By Lemma \ref{lem:round-spectrum}, the index set $\{l^2,\cdots,l^2+2l\}$ is exactly the degree-$l$ round cluster. These blocks, for $l\geq1$, exhaust all positive indices, so Theorem \ref{thm:comparison-expanded} is precisely Theorem \ref{thm:two-sided}.

\section{Proof of Theorem \ref{thm:main}}\label{sec:squeeze}
We first recall Calabi's classification theorem \sourcecite{Calabi}{Theorem 5.2}.
\begin{theorem}[\sourcecite{Calabi}{Theorem 5.2}]\label{calabi}
	Let $M$ be a $2$-sphere with a Riemannian metric with constant curvature $K$, and let $f:M\to\Sph^N(1/r^2)\subset\mathbb{R}^{N+1}$ be an isometric, minimal immersion of $M$ into the sphere of radius $r$, such that the image is not contained in any hyperplane of $\mathbb{R}^{N+1}$. Then
	\begin{enumerate}
		\item there exists an integer $s\ge 1$ such that $N=2s$ and the value of $K$ is uniquely determined at the value $$K=\frac{2}{s(s+1)r^2};$$
		\item the immersion $f$ is uniquely determined up to a rigid rotation of $\mathbb{S}^N$, and the $N+1$ components of the vector $f$ are a suitably normalized basis for the spherical harmonics of order $s$ on $M$.
	\end{enumerate}
\end{theorem}
We now prove a spectral consequence of Theorem \ref{thm:two-sided}.
\begin{proposition}\label{prop:squeeze}
	Let $s\geq1$ and let $g$ be a smooth metric on $\Sph^2$ satisfying
	\begin{equation*}
		\frac1{T(s+1)}\leq K_g\leq\frac1{T(s)}.
	\end{equation*}
	If $2\in\Spec(-\Delta_g)$, then
	\begin{equation*}
		K_g\equiv\frac{1}{T(s)}\quad\text{or}\quad K_g\equiv\frac{1}{T(s+1)}.
	\end{equation*}
\end{proposition}
\begin{proof}
	Assume that neither endpoint identity holds. Scale at the upper endpoint:
	\begin{equation*}
		g_{\mathrm{top}}=\frac1{T(s)}g.
	\end{equation*}
	By Lemma \ref{lem:scaling}, $K_{g_{\mathrm{top}}}=T(s)K_g\leq1$ and $\lambda_i(g_{\mathrm{top}})=T(s)\lambda_i(g)$. The curvature is not identically one. Define
	\begin{equation*}
		J_s=(s+1)^2.
	\end{equation*}
	Then $J_s-1=(s+1)^2-1$ is the last index in the round degree-$s$ cluster, while $J_s$ is the first position in the degree-$(s+1)$ cluster. The corresponding round eigenvalue is $s(s+1)=2T(s)$. Then Theorem \ref{thm:comparison-expanded} gives
	\begin{equation*}
		T(s)\lambda_{J_s-1}(g)=\lambda_{J_s-1}(g_{\mathrm{top}})<2T(s),
	\end{equation*}
	so
	\begin{equation*}
		\lambda_{J_s-1}(g)<2.
	\end{equation*}
	At the lower endpoint put
	\begin{equation*}
		g_{\mathrm{bot}}=\frac1{T(s+1)}g.
	\end{equation*}
	Then $K_{g_{\mathrm{bot}}}=T(s+1)K_g\geq1$, it is not identically one, and $\lambda_i(g_{\mathrm{bot}})=T(s+1)\lambda_i(g)$. The first position in the round degree-$(s+1)$ cluster is $J_s$, and its value is $(s+1)(s+2)=2T(s+1)$. The strict forward comparison gives
	\begin{equation*}
		T(s+1)\lambda_{J_s}(g)=\lambda_{J_s}(g_{\mathrm{bot}})>2T(s+1),
	\end{equation*}
	hence
	\begin{equation*}
		\lambda_{J_s}(g)>2.
	\end{equation*}
	Together,
	\begin{equation*}
		\lambda_{(s+1)^2-1}(g)<2<\lambda_{(s+1)^2}(g).
	\end{equation*}
	Every eigenvalue with index at most $(s+1)^2-1$ is below $2$, whereas every eigenvalue with index at least $(s+1)^2$ is above $2$. Thus $2\notin\Spec(-\Delta_g)$, a contradiction.
\end{proof}

\begin{proof}[\textbf{Proof of Theorem \ref{thm:main}}]
	Let $\pi\colon\widehat M\to M$ be the oriented cover from Lemma \ref{lem:cover}. The pinching lower bound is positive, so $\widehat M\cong\Sph^2$. Write $\widehat g=\pi^*g$ and lift the immersion. The lifted curvature satisfies
	\begin{equation*}
		\frac1{T(s+1)}\leq K_{\widehat g}\leq\frac1{T(s)}.
	\end{equation*}
	By Lemma \ref{lem:Takahashi}, $2\in\Spec(-\Delta_{\widehat g})$. Proposition \ref{prop:squeeze} therefore forces
	\begin{equation*}
		K_{\widehat g}\equiv T(s)^{-1}\quad\text{or}\quad K_{\widehat g}\equiv T(s+1)^{-1}.
	\end{equation*}
	Since $\pi$ is surjective and curvature pulls back, the same endpoint identity holds on $M$. Thus, by Theorem \ref{calabi}, the lifted immersion on the oriented double cover is one of Calabi's $2$-spheres. The dimension of the ambient space $N=2s+2$ or $N=2s$, respectively, if the immersion is linearly full.
\end{proof}

\section{Proof of Theorem \ref{thm:pmc-gaps}}
\label{sec:pmc-extension}

We now prove Theorem \ref{thm:pmc-gaps}. Yau's classification theorem \cite[Theorem 4]{YauPMC} reduces the case $H>0$ to a minimal immersion in a small round sphere, except for a three-dimensional CMC branch treated by the Hopf theorem.

\begin{lemma}\label{lem:pmc-gauss}
	For a surface in the unit sphere,
	\begin{equation}\label{eq:pmc-gauss}
		2K=2+4H^2-|A|^2=2\kappa-|A^\circ|^2,\quad|A|^2=|A^\circ|^2+2H^2.
	\end{equation}
\end{lemma}
\begin{proof}
	The traced Gauss equation for a surface in a unit sphere is
	\begin{equation*}
		2K=2+|\operatorname{tr}_gA|^2-|A|^2=2+4H^2-|A|^2.
	\end{equation*}
	Since $A=A^\circ+g\otimes\mathbf H$, the two summands are orthogonal: the first is trace-free, while $\langle A^\circ,g\otimes\mathbf H\rangle=\langle\operatorname{tr}_gA^\circ,\mathbf H\rangle=0$. Moreover $|g\otimes\mathbf H|^2=2H^2$. This proves \eqref{eq:pmc-gauss}.
\end{proof}

\begin{lemma}\label{lem:pmc-small-sphere}
	Let $x\colon M^2\to\Sph^N(1)$ have nonzero parallel mean curvature vector, and let $A_{\mathbf H}$ be the shape operator defined by $\langle A_{\mathbf H}X,Y\rangle=\langle A(X,Y),\mathbf H\rangle$. Assume that
	\begin{equation*}
		A_{\mathbf H}=H^2\Id.
	\end{equation*}
	This condition holds, in particular, when $M$ is minimal in a totally umbilical hypersurface of $\Sph^N(1)$. Then the $\R^{N+1}$-valued field
	\begin{equation}\label{eq:pmc-center}
		a=\frac{H^2x+\mathbf H}{\kappa}
	\end{equation}
	is constant and satisfies
	\begin{equation}\label{eq:pmc-center-norms}
		|a|^2=\frac{H^2}{\kappa},\quad|x-a|^2=\frac1\kappa,\quad\langle a,x-a\rangle=0.
	\end{equation}
	Consequently,
	\begin{equation*}
		y=\sqrt\kappa(x-a)\colon(M,\bar g=\kappa g)\to\Sph^N\cap a^\perp\cong\Sph^{N-1}
	\end{equation*}
	is a minimal immersion, and
	\begin{equation}\label{eq:pmc-normalized-curvature}
		K_{\bar g}=\frac{K_g}{\kappa}.
	\end{equation}
\end{lemma}
\begin{proof}
	We first verify the final assertion in the hypothesis. Suppose that $M$ is minimal in a totally umbilical hypersurface $Q\subset\Sph^N$. If $A^{M/Q}$ is the second fundamental form in $Q$ and $B^Q(U,V)=\mu\langle U,V\rangle\xi$ is the second fundamental form of $Q$ in $\Sph^N$, then the composition formula gives
	\begin{equation*}
		A=A^{M/Q}+\mu g\otimes\xi,\quad\mathbf H=\mu\xi.
	\end{equation*}
	Since $A^{M/Q}$ is trace free and orthogonal to $\xi$, it follows that $A_{\mathbf H}=\mu^2\Id=H^2\Id$. We now use only the assumptions stated in the lemma. Since $\nabla^\perp\mathbf H=0$, $X(H^2)=2\langle\nabla_X^\perp\mathbf H,\mathbf H\rangle=0$; hence $H$ and $\kappa$ are constant.  With $D$ denoting the Euclidean connection, the Weingarten formula and $\langle X,\mathbf H\rangle=0$ give
	\begin{equation*}
		D_X\mathbf H=-A_{\mathbf H}X+\nabla_X^\perp\mathbf H=-H^2X.
	\end{equation*}
	Differentiating \eqref{eq:pmc-center} now gives $D_Xa=0$. Because $\langle x,\mathbf H\rangle=0$ and $|x|=1$, one has
	\begin{equation*}
		x-a=\frac{x-\mathbf H}{\kappa},
	\end{equation*}
	and the identities in \eqref{eq:pmc-center-norms} follow. Thus $y$ takes values in the stated unit sphere and $y^*g_{\Sph^{N-1}}=\kappa g$. To prove minimality, recall that the Euclidean vector Laplacian of an arbitrary surface in the unit sphere satisfies
	\begin{equation*}
		\Delta_gx=2\mathbf H-2x.
	\end{equation*}
	Since $a$ is constant, we obtain
	\begin{equation*}
		\Delta_g y=\sqrt\kappa\,\Delta_g(x-a)=\sqrt\kappa(2\mathbf H-2x)=-2\kappa y,\quad \Delta_{\bar g}y=\kappa^{-1}\Delta_gy=-2y.
	\end{equation*}
	Lemma \ref{lem:Takahashi}, equivalently the vector-valued composition formula, shows that $y$ is minimal in the unit sphere. Finally, Lemma \ref{lem:scaling} gives \eqref{eq:pmc-normalized-curvature}.
\end{proof}

We also need Yau's classification theorem and the genus-zero Hopf theorem.
\begin{theorem}[Yau's classification theorem \sourcecite{YauPMC}{Theorem~4}]\label{thm:Yau-PMC-classification}
	If $M^2$ is a surface with nonzero parallel mean curvature vector in a real space form $\mathcal{N}$, then either $M^2$ is minimal in a totally umbilical hypersurface of $\mathcal{N}$ or $M^2$ lies in a three-dimensional totally umbilical submanifold of $\mathcal{N}$ and has constant mean curvature therein.
\end{theorem}

We also mention that B. Y. Chen independently obtained important classification results for surfaces with parallel mean curvature vector \cite{Chen73a,Chen73b,Chen10}.

\begin{lemma}[Hopf's genus-zero CMC theorem]\label{lem:pmc-Hopf}
	A constant-mean-curvature immersion of an oriented two-sphere into a three-dimensional real space form is totally umbilical.
\end{lemma}

\begin{proof}
	Let $A^{M/Q}$ denote the second fundamental form of the immersion $M\to Q^3$. Choose a global unit normal $\nu$ and write
	\begin{equation*}
		A^{M/Q}(X,Y)=h(X,Y)\nu,
	\end{equation*}
	where $h$ is the scalar second fundamental form. Set
	\begin{equation*}
		H_Q=\frac12\operatorname{tr}_g h,\quad h^\circ=h-H_Qg.
	\end{equation*}
	In a local conformal coordinate $z$, the Hopf quadratic differential is
	\begin{equation*}
		\Phi=h(\partial_z,\partial_z)dz^2=h^\circ(\partial_z,\partial_z)dz^2,
	\end{equation*}
	where the second equality follows from $g(\partial_z,\partial_z)=0$. The Codazzi equation in a space form gives $\bar\partial\Phi=0$ when $H_Q$ is constant. Hence $\Phi$ is a holomorphic section of $\mathcal K_M^{\otimes2}$. Since $M\cong\Sph^2$, this line bundle has degree $-4$ and therefore has no nonzero holomorphic section; equivalently, this is the $H^0(\CP^1,\mathcal O(-4))=0$ case of Lemma \ref{lem:CP1-cohomology}. Thus $\Phi=0$. A real symmetric trace-free two-tensor on a surface is completely determined by its $(2,0)$ part, since its $(0,2)$ part is the complex conjugate and its $(1,1)$ part is its trace. Therefore $\Phi=0$ implies $h^\circ=0$, or equivalently,
	\begin{equation*}
		\left(A^{M/Q}\right)^\circ=0.
	\end{equation*}
	Hence the immersion is totally umbilical in $Q^3$.
\end{proof}

Now we can give the proof of Theorem \ref{thm:pmc-gaps}.

\begin{proof}[\textbf{Proof of Theorem \ref{thm:pmc-gaps}}]
	Since $\nabla^\perp\mathbf H=0$, the functions $H$ and $\kappa$ are constant. If $H=0$, then $\kappa=1$, and the conclusion follows directly from Theorem \ref{thm:main}. Hence we assume from now on that $H>0$. The lower bound in \eqref{eq:pmc-K-pinching} is positive. Let $\pi\colon\widehat M\to M$ be the oriented double cover, write $\widehat g=\pi^*g$, and lift the immersion and its normal data. Then $\widehat M\cong\Sph^2$. We apply Yau's classification theorem to the lifted immersion. In Yau's first alternative, $\widehat M$ is minimal in a totally umbilical hypersphere. Lemma \ref{lem:pmc-small-sphere} gives a closed minimal immersion
	\begin{equation*}
		y\colon(\widehat M,\bar g=\kappa\widehat g)\to\Sph^{N-1}.
	\end{equation*}
	Its Gaussian curvature satisfies
	\begin{equation*}
		\frac{2}{(s+1)(s+2)}\leq K_{\bar g}=\frac{K_{\widehat g}}{\kappa}\leq\frac{2}{s(s+1)}.
	\end{equation*}
	Applying Theorem \ref{thm:main} to $y$ gives
	\begin{equation*}
		K_{\widehat g}\equiv\frac{2\kappa}{(s+1)(s+2)}\quad\text{or}\quad K_{\widehat g}\equiv\frac{2\kappa}{s(s+1)}.
	\end{equation*}
	In Yau's second alternative, $\widehat M$ is a surface with constant mean curvature vector in a three-dimensional totally umbilical submanifold $Q^3$. Write
	\begin{equation*}
		B^Q(U,V)=\widehat g(U,V)\zeta.
	\end{equation*}
	The Codazzi equation gives $\nabla^\perp\zeta=0$, and the Gauss equation gives
	\begin{equation*}
		\operatorname{sec}_Q=1+|\zeta|^2.
	\end{equation*}
	Thus $Q^3$ is a three-dimensional space form. Since $\widehat M\cong\Sph^2$, its normal line bundle in $Q^3$ is trivial, and Lemma \ref{lem:pmc-Hopf} gives
	\begin{equation*}
		\left(A^{\widehat M/Q}\right)^\circ=0.
	\end{equation*}
	Because $Q^3$ is totally umbilical, the composition formula $A=A^{\widehat M/Q}+B^Q|_{T\widehat M}$ gives, after taking trace-free parts,
	\begin{equation*}
		A^\circ=\left(A^{\widehat M/Q}\right)^\circ=0.
	\end{equation*}
	Lemma \ref{lem:pmc-gauss} therefore yields $K_{\widehat g}=\kappa$. For $s=1$, this is exactly the upper endpoint $2\kappa/[s(s+1)]$. For $s\geq2$, it contradicts
	\begin{equation*}
		K_{\widehat g}\leq\frac{2\kappa}{s(s+1)}<\kappa.
	\end{equation*}
	Thus this alternative can occur only for $s=1$, and it also gives an endpoint value. In either alternative, $K_{\widehat g}$ is identically equal to one of the two endpoints. Since $K_{\widehat g}=K_g\circ\pi$ and $\pi$ is surjective, the same endpoint identity holds on $M$.
\end{proof}

\section{Proof of Theorem \ref{thm:higher-rotational-gap}}\label{sec:higher-rotational}

Throughout this section $n\geq3$. A smooth rotationally symmetric metric on $S^n$ will be written globally as
\begin{equation*}
	h=dt^2+f(t)^2g_{\Sph^{n-1}},\quad 0\leq t\leq a,
\end{equation*}
where $f>0$ on $(0,a)$ and the usual smooth pole compatibility conditions hold. In particular,
\begin{equation*}
	f(0)=f(a)=0,\quad f'(0)=1,\quad f'(a)=-1,
\end{equation*}
and the Taylor expansions at the two poles have the parity required for the warped product to extend smoothly across them. For $0<t<a$, a $2$-plane is called radial if it contains $\partial_t$, and tangential if it is contained in $T(\{t\}\times S^{n-1})$. By rotational symmetry, all radial planes at the same point have the same sectional curvature, and so do all tangential planes. We denote these two curvatures by $K_{\mathrm{rad}}$ and $K_{\mathrm{tan}}$, respectively. They are given by
\begin{equation*}
	K_{\mathrm{rad}}=-\frac{f''}{f},\quad K_{\mathrm{tan}}=\frac{1-(f')^2}{f^2}.
\end{equation*}
Both extend smoothly to the poles. We first establish the reverse of the rotational comparison theorem of Lin, Wang, and Xu \cite{LinWangXu}. Let
\begin{equation*}
	\mathscr H_{\rm rad}=L^2\left((0,a),f^{n-1}dt;\R\right).
\end{equation*}
For $m\geq0$, let $\mathcal D_m$ be the smooth radial core consisting of those functions $u\in C^\infty((0,a);\R)$ for which there exist smooth real-valued functions $\psi_0$ and $\psi_a$, defined in a neighborhood of $0$, such that
\begin{equation*}
	u(t)=t^m\psi_0(t^2)\text{ near }t=0,\quad u(t)=(a-t)^m\psi_a((a-t)^2)\text{ near }t=a.
\end{equation*}
On this core define
\begin{equation}\label{eq:radial-B}
	\begin{aligned}
		\ell_m[u]&=\int_0^a\left(|u'|^2+\frac{m(m+n-2)}{f^2}|u|^2\right)f^{n-1}dt,\\
		B_mu&=u'-m\frac{f'}f u,\\
		B_m^\dagger v&=-v'-(m+n-1)\frac{f'}fv.
	\end{aligned}
\end{equation}
Let $L_m$ be the self-adjoint operator associated with the closure of $\ell_m$, and let $B_m$ denote the graph closure of the first-order operator. Set
\begin{equation*}
	\alpha_m=m(m+n-1),
\end{equation*}
and define the closed forms
\begin{align*}
	a_m[u]&=\|B_mu\|^2+\alpha_m\|u\|^2,\quad\Dom(a_m)=\Dom(B_m),\\
	c_m[v]&=\|B_m^*v\|^2+\alpha_m\|v\|^2,\quad\Dom(c_m)=\Dom(B_m^*).
\end{align*}
Their associated operators are denoted by $A_m$ and $C_m$.
\begin{lemma}[\sourcecite{LinWangXu}{Definitions 3.4 and 3.5; Lemmas 4.6, 4.7, 5.2, and 5.4}]\label{lem:radial-defects}
	For every $m\geq0$, the relevant form domains agree:
	\begin{equation}\label{eq:radial-domains}
		\Dom(\ell_m)=\Dom(a_m),\quad\Dom(c_m)=\Dom(a_{m+1})=\Dom(\ell_{m+1}).
	\end{equation}
	The common domains embed compactly into $\mathscr H_{\rm rad}$, and
	\begin{align}
		c_m[v]-a_{m+1}[v]&= (2m+n)\int_0^a(K_{\rm rad}-1)|v|^2f^{n-1}dt,\label{eq:radial-partner-defect}\\
		\ell_m[u]-a_m[u]&=m\int_0^a\left((m+n-2)(K_{\rm tan}-1)+(K_{\rm rad}-1)\right)|u|^2f^{n-1}dt.\label{eq:radial-laplacian-defect}
	\end{align}
	Moreover,
	\begin{equation}\label{eq:radial-kernels}
		\ker B_m=\R f^m,\quad\ker B_m^*=\{0\}.
	\end{equation}
	Consequently, for $j\ge0$,
	\begin{equation}\label{eq:radial-partner-shift}
		\lambda_0(A_m)=\alpha_m,\quad\lambda_{j+1}(A_m)=\lambda_j(C_m).
	\end{equation}
\end{lemma}

\begin{proof}
	This is the radial construction of Lin, Wang, and Xu \cite[Sections 3-5]{LinWangXu}, with their dimensional parameter replaced by $n+1$. We spell out the points whose validity is independent of the sign of the curvature defect. Put $q=f'/f$. On the pole-compatible cores, the formal adjoint in \eqref{eq:radial-B} gives the identities
	\begin{align*}
		B_m^*B_m&=-\frac{d^2}{dt^2}-(n-1)q\frac d{dt}+m(m+n-2)q^2+m\frac{f''}{f},\\
		B_mB_m^*&=-\frac{d^2}{dt^2}-(n-1)q\frac d{dt}+(m+1)(m+n-1)q^2-(m+n-1)\frac{f''}{f}.
	\end{align*}
	Using $f''/f=-K_{\rm rad}$, $(1-(f')^2)/f^2=K_{\rm tan}$, and $\alpha_{m+1}-\alpha_m=2m+n$, these formulas give on $\mathscr D_{m+1}$ and $\mathscr D_m$, respectively,
	\begin{align*}
		c_m-a_{m+1}&=(2m+n)(K_{\rm rad}-1),\\
		\ell_m-a_m&=m\bigl((m+n-2)(K_{\rm tan}-1)+(K_{\rm rad}-1)\bigr),
	\end{align*}
	in the quadratic-form sense, which proves \eqref{eq:radial-partner-defect}-\eqref{eq:radial-laplacian-defect} on the cores. We next justify passage to the closed forms. The pole-core theorem for the adjoint and the graph-domain identification of Lin, Wang, and Xu \cite[Lemmas 4.6 and 4.7]{LinWangXu}, after the same dimensional substitution, give
	\begin{equation*}
		\Dom(B_m^*)=\overline{\mathscr D_{m+1}}^{\|\cdot\|_{B_m^*}},\quad\Dom(B_m)=\Dom(\ell_m).
	\end{equation*}
	The two displayed defect potentials are smooth and bounded up to the poles. Consequently the $c_m$- and $a_{m+1}$-form norms are equivalent on $\mathscr D_{m+1}$, while the $a_m$- and $\ell_m$-form norms are equivalent on $\mathscr D_m$. Taking the corresponding completions proves \eqref{eq:radial-domains} and extends both defect identities to their common closed domains. This is precisely the curvature-sign-free part of \cite[Lemma 5.2]{LinWangXu}; the sign of $K_{\rm rad}-1$ is used there only later, when the form order is imposed. The radial embedding into the appropriate spherical-harmonic summand of $H^1(\Sph^n,h)$ then gives the compact form-domain embeddings by Rellich's theorem \cite[Lemmas 4.3 and 5.2]{LinWangXu}. If $B_mu=0$, interior distributional regularity gives $(f^{-m}u)'=0$ on $(0,a)$, so $u=cf^m$. The pole expansions of $f$ show that $f^m\in\mathscr D_m\subset\Dom(B_m)$. If $B_m^*v=0$, the adjoint-core description above shows that the formal adjoint equation holds distributionally in the interior, and hence $$v=cf^{-(m+n-1)}.$$ Near a pole, its squared $\mathscr H_{\rm rad}$-density is comparable to $t^{-2m-n+1}\,dt$, which is not integrable for $n\geq3$. Thus $c=0$, and \eqref{eq:radial-kernels} follows. Finally, for any closed densely defined operator, the positive spectra of $B_m^*B_m$ and $B_mB_m^*$ agree with multiplicity, with the eigenspace isomorphism induced by $B_m$. Since $\ker B_m$ is one-dimensional and $\ker B_m^*=0$, the eigenvalue $\alpha_m$ occurs as one unpaired bottom eigenvalue of $A_m$. Removing it gives \eqref{eq:radial-partner-shift}.
\end{proof}

\begin{lemma}[\sourcecite{LinWangXu}{Lemma 6.3}]\label{lem:strict-radial-minmax}
	Let $s$ and $t$ be lower-bounded closed forms in $\mathscr H_{\rm rad}$ with a common form domain $V$ compactly embedded in $\mathscr H_{\rm rad}$, and let $S$ and $T$ be their associated self-adjoint operators. Suppose
	\begin{equation*}
		s[u]=t[u]+\int_0^a P|u|^2f^{n-1}dt,\quad u\in V,	
	\end{equation*}
	where $P\geq0$ is bounded and is positive on a nonempty open interval $U\subset(0,a)$. Assume that every eigenfunction of $S$ is smooth in the interior and satisfies a scalar second-order ordinary differential equation with smooth coefficients. Then
	\begin{equation*}
		\lambda_j(S)>\lambda_j(T),\quad j\geq0.
	\end{equation*}
\end{lemma}

\begin{proof}
	Let $F_j$ be the span of an orthonormal list of the first $j+1$ eigenfunctions of $S$, counted with multiplicity. If a nonzero $u\in F_j$ vanished on $U$, decompose it into components belonging to the distinct eigenvalues of $S$.  Applying powers of the local differential expression of $S$ on $U$ and using the resulting Vandermonde system shows that each eigencomponent vanishes on $U$. Ordinary differential equation uniqueness then forces every component to vanish identically, a contradiction. Hence
	\begin{equation*}
		\int_0^aP|u|^2f^{n-1}\,dt>0\quad(0\ne u\in F_j).
	\end{equation*}
	Compactness of the unit sphere of $F_j$ makes the defect uniformly positive there. Since $\dim F_j=j+1$ and the $S$-Rayleigh quotient is at most $\lambda_j(S)$ on $F_j$, the min-max principle for $T$ gives $\lambda_j(T)<\lambda_j(S)$. This is the radial version of Lin, Wang, and Xu \cite[Lemma 6.3]{LinWangXu}.
\end{proof}

\begin{theorem}\label{thm:reverse-rotational}
	Let $h$ be a smooth rotationally symmetric metric on $S^n$ and let $g_{\rd,n}$ be the unit round metric. If $\sec_h\leq1$, then
	\begin{equation}\label{eq:reverse-rotational-comparison}
		\lambda_i(h)\leq\lambda_i(g_{\rd,n}),\quad i\geq1.
	\end{equation}
	Moreover, if equality holds for some $i_0\ge1$, then $(S^n,h)$ is isometric to $\Sph^n$.
\end{theorem}
\begin{proof}
	The curvature assumption gives $K_{\rm rad}\leq1$ and $K_{\rm tan}\leq1$. Hence \eqref{eq:radial-partner-defect} and \eqref{eq:radial-laplacian-defect} imply the form orders
	\begin{equation*}
		C_m\leq A_{m+1},\quad L_m\leq A_m.
	\end{equation*}
	By min-max and \eqref{eq:radial-partner-shift}, $$\lambda_{j+1}(A_m)=\lambda_j(C_m)\leq\lambda_j(A_{m+1}).$$ Iterating this inequality until the bottom eigenvalue of $A_{m+j}$ and using $\lambda_0(A_q)=\alpha_q$ gives
	\begin{equation}\label{eq:radial-eigen-bound}
		\lambda_j(A_m)\leq\alpha_{m+j},\quad\lambda_j(L_m)\leq\alpha_{m+j}.
	\end{equation}
	Suppose now that $h$ is not round. If $K_{\rm rad}\equiv1$, then $f''+f=0$ with $f(0)=0$ and $f'(0)=1$, so $f(t)=\sin t$. Since $f>0$ on $(0,a)$ and vanishes at the second pole, necessarily $a=\pi$, and the metric is unit round, a contradiction. Thus there is a nonempty open interval $U$ on which $K_{\rm rad}<1$. Rewrite \eqref{eq:radial-partner-defect} as $$a_{m+1}[v]=c_m[v]+(2m+n)\int_0^a(1-K_{\rm rad})|v|^2f^{n-1}\,dt.$$ The operators $A_{m+1}$ are bounded-potential perturbations of the radial Sturm-Liouville operators $L_{m+1}$; hence their eigenfunctions are smooth in the interior and obey scalar second-order ODEs with smooth coefficients. Lemma \ref{lem:strict-radial-minmax}, with $S=A_{m+1}$ and $T=C_m$, therefore gives
	\begin{equation*}
		\lambda_j(C_m)<\lambda_j(A_{m+1}),\quad m,j\geq0.
	\end{equation*}
	Combining this with \eqref{eq:radial-partner-shift} and iterating yields
	\begin{equation*}
		\lambda_j(A_m)<\alpha_{m+j},\quad j\geq1.
	\end{equation*}
	For $m=0$, we have $L_0=A_0$. For $m\geq1$, the second defect identity can be written
	\begin{equation*}
		a_m[u]=\ell_m[u]+m\int_0^a\left((m+n-2)(1-K_{\mathrm tan})+(1-K_{\mathrm rad})\right)|u|^2f^{n-1}dt.
	\end{equation*}
	The potential is nonnegative everywhere and positive on $U$. Applying Lemma \ref{lem:strict-radial-minmax} with $S=A_m$ and $T=L_m$ gives $\lambda_j(L_m)<\lambda_j(A_m)$ for every $j\geq0$. Consequently,
	\begin{equation}\label{eq:strict-radial-L}
		\lambda_j(L_m)<(m+j)(m+j+n-1)\quad\text{whenever }(m,j)\ne(0,0).
	\end{equation}
	Let $d_{n-1,m}$ denote the real dimension of the degree-$m$ spherical harmonics on $\Sph^{n-1}$. The closed Dirichlet form decomposes into its angular spherical-harmonic summands, and hence
	\begin{equation*}
		-\Delta_h\simeq\bigoplus_{m=0}^{\infty}L_m^{\oplus d_{n-1,m}};
	\end{equation*}
	see Lin, Wang, and Xu \cite[Lemma 4.4]{LinWangXu}, again with their parameter $n$ replaced by $n+1$. Fix $\ell\geq1$. For every pair $0\leq m\leq\ell$, $0\leq j\leq\ell-m$ other than $(m,j)=(0,0)$, equation \eqref{eq:strict-radial-L} gives
	\begin{equation*}
		\lambda_j(L_m)<(m+j)(m+j+n-1)\leq\mu_{n,\ell}.
	\end{equation*}
	The exceptional pair is the constant zero mode of $L_0$, and $0<\mu_{n,\ell}$. Thus all these modes are strictly below $\mu_{n,\ell}$, with total multiplicity at least
	\begin{equation*}
		\sum_{m=0}^{\ell}d_{n-1,m}(\ell-m+1).
	\end{equation*}
	The spherical-harmonic branching identity
	\begin{equation*}
		d_{n,k}=\sum_{m=0}^k d_{n-1,m},
	\end{equation*}
	proved in Lin, Wang, and Xu \cite[Lemma 6.1]{LinWangXu} after the same dimensional substitution, gives
	\begin{equation*}
		\sum_{m=0}^{\ell}d_{n-1,m}(\ell-m+1)=\sum_{k=0}^{\ell}d_{n,k}\eqqcolon J_{n,\ell}.
	\end{equation*}
	Therefore
	\begin{equation*}
		\lambda_{J_{n,\ell}-1}(h)<\mu_{n,\ell}.
	\end{equation*}
	Every index in the degree-$\ell$ round cluster is at most $J_{n,\ell}-1$, so $\lambda_i(h)<\lambda_i(g_{\rd,n})$ throughout that cluster.  Since the positive round clusters exhaust all positive indices, the strict assertion follows. For the non-strict assertion, use \eqref{eq:radial-eigen-bound} instead of \eqref{eq:strict-radial-L}. It gives at least $J_{n,\ell}$ eigenvalues less than or equal to $\mu_{n,\ell}$, hence $\lambda_{J_{n,\ell}-1}(h)\leq\mu_{n,\ell}$ and the same clusterwise argument proves \eqref{eq:reverse-rotational-comparison}.
\end{proof}

For later indexing, let
\begin{equation*}
	d_{n,k}=\dim\mathcal H_k(\Sph^n),\quad J_{n,r}=\sum_{k=0}^rd_{n,k}.
\end{equation*}
Thus $J_{n,r}-1$ is the last zero-based position of the round degree-$r$ cluster and $J_{n,r}$ is the first position of the degree-$(r+1)$ cluster:
\begin{equation}\label{eq:higher-round-boundary}
	\lambda_{J_{n,r}-1}(g_{\rd,n})=\mu_{n,r},\quad\lambda_{J_{n,r}}(g_{\rd,n})=\mu_{n,r+1}.
\end{equation}

\begin{proof}[\textbf{Proof of Theorem \ref{thm:higher-rotational-gap}}]
	Write $\kappa_r=\kappa_{n,r}$ and $\kappa_{r+1}=\kappa_{n,r+1}$, and form the two rescaled metrics
	\begin{equation*}
		g_+=\kappa_r g,\quad g_-=\kappa_{r+1}g.
	\end{equation*}
	The sectional-curvature upper bound in \eqref{eq:higher-pinching} and Lemma \ref{lem:scaling} give
	\begin{equation*}
		\sec_{g_+}\leq1.
	\end{equation*}
	The Ricci lower bound in \eqref{eq:higher-pinching}, together with the Ricci scaling in Lemma \ref{lem:scaling}, gives
	\begin{equation*}
		\operatorname{Ric}_{g_-}=\operatorname{Ric}_g\geq(n-1)\kappa_{r+1}g=(n-1)g_-.
	\end{equation*}
	If $g_+$ is unit round, then $\sec_g\equiv\kappa_r$; if $g_-$ is unit round, then $\sec_g\equiv\kappa_{r+1}$. Assume for contradiction that neither metric is round. The strict reverse comparison proved in Theorem \ref{thm:reverse-rotational}, together with \eqref{eq:higher-round-boundary}, gives
	\begin{equation*}
		\frac{\lambda_{J_{n,r}-1}(g)}{\kappa_r}=\lambda_{J_{n,r}-1}(g_+)<\mu_{n,r}.
	\end{equation*}
	Since $\kappa_r\mu_{n,r}=n$,
	\begin{equation}\label{eq:higher-below-n}
		\lambda_{J_{n,r}-1}(g)<n.
	\end{equation}
	For the lower-side comparison we invoke Lin, Wang, and Xu \cite[Theorem 6.5]{LinWangXu}, with their dimension parameter replaced by $n+1$, which applies directly to the rotationally symmetric metric $g_-$ because $\operatorname{Ric}_{g_-}\geq(n-1)g_-$. It gives the ordered comparison with the unit round sphere, and its rigidity statement says that equality at any positive index forces $g_-$ to be unit round. Since $g_-$ is assumed nonround, the inequality at the positive index $J_{n,r}$ is strict:
	\begin{equation*}
		\frac{\lambda_{J_{n,r}}(g)}{\kappa_{r+1}}=\lambda_{J_{n,r}}(g_-)>\lambda_{J_{n,r}}(g_{\rd,n})=\mu_{n,r+1}.
	\end{equation*}
	Because $\kappa_{r+1}\mu_{n,r+1}=n$,
	\begin{equation}\label{eq:higher-above-n}
		\lambda_{J_{n,r}}(g)>n.
	\end{equation}
	Equations \eqref{eq:higher-below-n} and \eqref{eq:higher-above-n} place $n$ strictly between two consecutive ordered eigenvalues, so $$n\notin\Spec(-\Delta_g).$$ This contradicts Lemma \ref{lem:Takahashi}, since the immersion is minimal. Therefore one of $g_+$ and $g_-$ is round, proving \eqref{eq:higher-endpoints}.
\end{proof}

We conclude by formulating the following higher-dimensional gap conjecture.
\begin{conjecture}\label{conj:higher-gap-no-symmetry}
	Let $n\geq3$ and $r\geq1$ be integers, and let $x\colon (S^n,g)\to\Sph^N$ be a closed minimal immersion. Suppose that the sectional curvature $\sec_g$ of $S^n$ satisfies
	\begin{equation*}
		\kappa_{n,r+1}=\frac{n}{(r+1)(r+n)}\leq \sec_g\leq\frac{n}{r(r+n-1)}=\kappa_{n,r}.
	\end{equation*}
	Then
	\begin{equation*}
		\sec_g\equiv\kappa_{n,r}\quad\text{or}\quad\sec_g\equiv\kappa_{n,r+1}.
	\end{equation*}
\end{conjecture}

The first-gap case $r=1$ is already known in a stronger form: building on his earlier extrinsic characterization of real Veronese submanifolds \cite{Itoh1975}, T. Itoh \cite{Itoh1978} proved that the lower bound $\sec_g\geq\kappa_{n,2}=n/[2(n+1)]$ already forces $\sec_g\equiv\kappa_{n,1}=1$ or $\sec_g\equiv\kappa_{n,2}$.

\medskip
\noindent\textbf{Acknowledgements.}
J. Q. Ge is partially supported by NSFC (No. 12571049) and the Fundamental Research Funds for the Central Universities. F. G. Li is partially supported by NSFC (No. 12271040 and 12501061), the Guangdong Provincial Association for Science and Technology Youth Talent Support Program (No. SKXRC2026413), and the Research Start-up Funding of Beijing Institute of Technology (No. 5640011253301).


\begin{thebibliography}{99}

\bibitem{Aronszajn} N. Aronszajn, {\em A unique continuation theorem for solutions of elliptic partial differential equations or inequalities of second order}, J. Math. Pures Appl., {\bf 36} (1957), 235--249.

\bibitem{BKSS} K. Benko, M. Kothe, K. D. Semmler and U. Simon, {\em Eigenvalues of the Laplacian and curvature}, Colloq. Math., {\bf 42} (1979), 19--31.

\bibitem{Bolt88a} J. Bolton, G. R. Jensen, M. Rigoli and L. M. Woodward, {\em On conformal minimal immersions of $\mathbb{S}^2$ into $\CP^n$}, Math. Ann., {\bf 279} (1988), 599--620.

\bibitem{Bolt88b} J. Bolton and L. M. Woodward, {\em On the Simon conjecture for minimal immersions with $S^1$-symmetry}, Math. Z., {\bf 200} (1988), 111--121.

\bibitem{Calabi} E. Calabi, {\em Minimal immersions of surfaces in Euclidean spheres}, J. Differential Geom., {\bf 1} (1967), 111--125.

\bibitem{Chen73a} B. Y. Chen, {\em On the surface with parallel mean curvature vector}, Indiana Univ. Math. J., {\bf 22} (1973), 655--666.

\bibitem{Chen73b} B. Y. Chen, {\em Geometry of submanifolds}, Mercer Dekker, New York, 1973.

\bibitem{Chen10} B. Y. Chen, {\em Submanifolds with parallel mean curvature vector in Riemannian and indefinite space forms}, Arab J. Math. Sci., {\bf 16} (2010), 1--46.

\bibitem{DGL} W. R. Ding, J. Q. Ge and F. G. Li, {\em Pinching rigidity of minimal surfaces in spheres}, Sci. China Math., {\bf 68} (2025), 2189--2206.

\bibitem{DingGeLiThird} W. R. Ding, J. Q. Ge and F. G. Li, {\em On Simon's third gap conjecture for minimal surfaces in spheres}, arXiv:2603.03070.

\bibitem{DingGeLiPMC} W. R. Ding, J. Q. Ge and F. G. Li, {\em Pinching rigidity of surfaces with parallel mean curvature vector in spheres}, arXiv:2607.23428.

\bibitem{doCarmoWallach} M. P. do Carmo and N. R. Wallach, {\em Minimal immersions of spheres into spheres}, Ann. of Math. (2), {\bf 93} (1971), 43--62.

\bibitem{Forster} O. Forster, {\em Lectures on Riemann Surfaces}, Graduate Texts in Mathematics, vol. 81, Springer, New York, 1981.

\bibitem{GriffithsHarris} P. Griffiths and J. Harris, {\em Principles of Algebraic Geometry}, Wiley, New York, 1978.

\bibitem{GuXuXu} J. R. Gu, H. W. Xu, Z. Y. Xu, et al., {\em A survey on rigidity problems in geometry and topology of submanifolds}, In: Proceedings of the 6th International Congress of Chinese Mathematicians. Advanced Lectures in Mathematics, vol. 37. Beijing-Boston: Higher Education Press-International Press, (2016), 79--99.

\bibitem{Itoh1975} T. Itoh, {\em On Veronese manifolds}, J. Math. Soc. Japan, {\bf 27} (1975), 497--506.

\bibitem{Itoh1978} T. Itoh, {\em Addendum to my paper ``On Veronese manifolds''}, J. Math. Soc. Japan, {\bf 30} (1978), 73--74.

\bibitem{It88} T. Itoh, {\em A characterization of the generalized Veronese surfaces}, Proc. Amer. Math. Soc., {\bf 104} (1988), 571--576.

\bibitem{JostRiemann} J. Jost, {\em Compact Riemann Surfaces}, third ed., Springer, Berlin, 2006.

\bibitem{Kato} T. Kato, {\em Perturbation Theory for Linear Operators}, Springer, Berlin, 1995.

\bibitem{Simon} M. Kozlowski and U. Simon, {\em Minimal immersions of $2$-manifolds into spheres}, Math. Z., {\bf 186} (1984), 377--382.

\bibitem{LiSimon} H. Z. Li and U. Simon, {\em Quantization of curvature for compact surfaces in $S^n$}, Math. Z., {\bf 245} (2003), 201--216.

\bibitem{LinWangXu} S. J. Lin, H. B. Wang and G. Y. Xu, {\em Eigenvalues on spheres}, arXiv:2607.11544v1.

\bibitem{Simon1} U. Simon, {\em Eigenvalues of the Laplacian and minimal immersions into spheres}, in: Differential Geometry, Montreal: Pitman, Boston, MA, {\bf 131} (1985), 115--120.

\bibitem{Yau} M. Scherfner, S. Weiss and S. T. Yau, {\em A review of the Chern conjecture for isoparametric hypersurfaces in spheres}, in: Advances in Geometric Analysis, Advanced Lectures in Mathematics, vol. 21. Beijing-Boston: Higher Education Press-International Press, (2012), 175--187.

\bibitem{Takahashi} T. Takahashi, {\em Minimal immersions of Riemannian manifolds}, J. Math. Soc. Japan, {\bf 18} (1966), 380--385.

\bibitem{TaylorPDE} M. E. Taylor, {\em Partial Differential Equations II: Qualitative Studies of Linear Equations}, second ed., Applied Mathematical Sciences, vol. 116, Springer, New York, 2011.

\bibitem{Wells} R. O. Wells, Jr., {\em Differential Analysis on Complex Manifolds}, third ed., Graduate Texts in Mathematics, vol.65, Springer, New York, 2008.

\bibitem{YauPMC} S. T. Yau, {\em Submanifolds with constant mean curvature I}, Amer. J. Math., {\bf 96} (1974), 346--366.

\end{thebibliography}
\end{document}